\documentclass[oneside,english]{amsart}
\usepackage[T1]{fontenc}
\usepackage[latin9]{inputenc}
\usepackage{geometry}
\usepackage{textcomp}
\usepackage{mathrsfs}
\usepackage{amstext}
\usepackage{amsthm}
\usepackage{amssymb}
\usepackage[all]{xy}

\makeatletter
\numberwithin{equation}{section}
\numberwithin{figure}{section}
\theoremstyle{plain}
\newtheorem{thm}{\protect\theoremname}[section]
  \theoremstyle{plain}
  \newtheorem{prop}[thm]{\protect\propositionname}
  \theoremstyle{plain}
  \newtheorem{lem}[thm]{\protect\lemmaname}
  \theoremstyle{definition}
  \newtheorem{defn}[thm]{\protect\definitionname}
  \theoremstyle{plain}
  \newtheorem{cor}[thm]{\protect\corollaryname}
  \theoremstyle{remark}
  \newtheorem*{rem*}{\protect\remarkname}

\usepackage[all]{xy}

\makeatother

\usepackage{babel}
  \providecommand{\corollaryname}{Corollary}
  \providecommand{\definitionname}{Definition}
  \providecommand{\lemmaname}{Lemma}
  \providecommand{\propositionname}{Proposition}
  \providecommand{\remarkname}{Remark}
\providecommand{\theoremname}{Theorem}

\begin{document}

\title{Induction and restriction of $(\varphi,\Gamma)$-modules}
\begin{abstract}
Let $L$ be a non-archimedean local field of characteristic 0. We
present a variant of the theory of $(\varphi,\Gamma)$-modules associated
with Lubin-Tate groups, developed by Kisin and Ren \cite{=00005BKi-Re=00005D},
in which we replace the Lubin-Tate tower by the maximal abelian extension
$\Gamma=Gal(L^{ab}/L).$ This variation allows us to compute the functors
of induction and restriction for $(\varphi,\Gamma)$-modules, when
the ground field $L$ changes. We also give a self-contained account
of the Cherbonnier-Colmez theorem on overconvergence in our setting.
\end{abstract}

\author{Ehud de Shalit and Gal Porat}

\address{Hebrew University, Jerusalem, Israel}

\email{ehud.deshalit@mail.huji.ac.il, galporat1@gmail.com}

\maketitle
Let $L$ be a finite extension of the field $\mathbb{Q}_{p}$. Let
$\mathcal{O}$ denote its ring of integers, $\kappa$ its residue
field, $q$ the cardinality of $\kappa,$ $G=Gal(\overline{L}/L)$
the absolute Galois group, and $\mathrm{Rep}{}_{\mathcal{O}}(L)$
the category of finitely generated $\mathcal{O}$-modules equipped
with a continuous $G$-action. In order to study $\mathrm{Rep_{\mathcal{O}}}(L)$
Fontaine introduced in \cite{=00005BFo=00005D} the category $\mathrm{Mod}{}_{\varphi,\Gamma}(\mathbf{A}_{L})$
of $(\varphi,\Gamma)$-modules over the ``period ring'' $\mathbf{A}_{L}$.
He then constructed a functorial equivalence between the rather complicated
category $\mathrm{Rep_{\mathcal{O}}}(L)$, and the seemingly simpler
full subcategory $\mathrm{Mod}{}_{\varphi,\Gamma}^{\acute{e}t}(\mathbf{A}_{L})$
of $\mathrm{Mod}{}_{\varphi,\Gamma}(\mathbf{A}_{L})$ consisting of
\emph{étale} $(\varphi,\Gamma)$-modules.

The ring $\mathbf{A}_{L}$ and the category $\mathrm{Mod}{}_{\varphi,\Gamma}(\mathbf{A}_{L})$
depend on the choice of a certain $\Gamma$-extension $L_{\infty}/L$,
which in Fontaine's work, and in most of the applications that followed,
was taken to be the cyclotomic $\mathbb{Z}_{p}$-extension of $L$.
Kisin and Ren \cite{=00005BKi-Re=00005D} introduced a variant, in
which $L_{\infty}/L$ is the extension generated by the torsion points
of a Lubin-Tate group $\mathcal{G}$ defined over $\mathcal{O}.$
Using this new variant they were able to generalize results of Berger,
Colmez and Wach on crystalline representations, which were previously
available only for $L$ unramified over $\mathbb{Q}_{p}$ and $L_{\infty}$
its cyclotomic extension.

A full account of the Kisin-Ren theory was provided in a recent monograph
\cite{=00005BSchn=00005D} of Schneider. Among other things, this
book substitutes Scholze's notion of ``tilting'' for the original
``field of norms'' approach. However, the constructions in both
these references depend strongly on the Lubin-Tate group law, and
in particular on the choice of a uniformizer of $L$. This makes it
difficult to express the functors of restriction and induction in
the language of $(\varphi,\Gamma)$-modules, when we let the base
field $L$ change. Similar difficulties, arising from the incompatibility
of Lubin-Tate theory with base-field extensions, were encountered
in Iwasawa's development of local class field theory \cite{=00005BIw=00005D},
and in the generalization of local class field theory to meta-abelian
extensions studied in \cite{=00005BK-dS=00005D}.

\bigskip{}

The goal of this note is to present yet another variant of the Kisin-Ren
theory, in which the extension $L_{\infty}/L$ is replaced by the
maximal abelian extension $L^{ab}/L.$ Thus $\Gamma=Gal(L^{ab}/L).$
As such, no choice of a uniformizer or a Lubin-Tate group is involved,
and the ambiguity mentioned above is resolved. For our purpose it
is also necessary to replace the fields of norms of Fontaine and Wintenberger
by their \emph{completed perfections }which we shall denote by the
letter $F$. More precisely, if the field of norms of a Lubin-Tate
tower is (non-canonically) isomorphic to $\kappa((\omega)),$ then
our $F$ will be isomorphic to the completion of the perfection of
$\overline{\kappa}((\omega))$. Intrinsically, $F$ is defined to
be the tilt $(\widehat{L^{ab}})^{\flat}$ of the completion $\widehat{L^{ab}}$
of $L^{ab}.$ Such ``complete perfections'' of norm-fields were
already shown to be useful in the work of Cherbonnier and Colmez \cite{=00005BCh-Co=00005D}
and Kedlaya and Liu \cite{=00005BKe-L=00005D}\cite{=00005BKe=00005D},
who nevertheless avoided the extension of scalars from $\kappa$ to
$\overline{\kappa}$, and did not consider these objects in the Lubin-Tate
setting. The coefficient ring for our $(\varphi,\Gamma)$-modules
is modified accordingly. As $F$ is perfect, one can simply take as
coefficients the ring $\widetilde{\mathrm{\mathbf{A}}}_{L}=W(F)_{L}$
of Witt vectors of $F$, tensored over $W(\kappa)$ with $\mathcal{O}$,
instead of the smaller $\mathrm{\mathbf{A}}_{L}$, whose construction
would require further work.

The structure of $\mathbf{\widetilde{A}}_{L}$ and the new category
$\mathrm{Mod}_{\varphi,\Gamma}(L)$ of $(\varphi,\Gamma)$-modules
over it are elucidated in §1. The main theorem on the equivalence
between $\mathrm{Rep}{}_{\mathcal{O}}(L)$ and the full subcategory
$\mathrm{Mod}_{\varphi,\Gamma}^{\acute{e}t}(L)$ is given in §2. These
two sections repeat well-known results. As we follow \cite{=00005BSchn=00005D}
closely, proofs are omitted. Once we have established the new variant,
the computations of the functors of induction and restriction are
straightforward, and are given in §3. The main result concerning these
functors is Theorem \ref{thm:Ind-Res}.

We remark that for ``cyclotomic $(\varphi,\Gamma)$-modules'', i.e.
when $\Gamma$ is the cyclotomic extension, a similar result was obtained
in Liu's thesis, in the framework of $(\varphi,\Gamma)$-modules over
the Robba ring, \emph{cf.} \cite{=00005BLiu=00005D}, Proposition
2.1. Recall that the Robba ring does not admit an integral structure,
and the appropriate $(\varphi,\Gamma)$-modules over the Robba ring
classify $L$-vector space representations of $G$ rather than $\mathcal{O}$-module
representations. More importantly, in the Kisin-Ren setting, the fields
of norms attached to two Lubin-Tate towers over $L_{1}$ and $L_{2}$,
where $L_{1}\subset L_{2}$ is a finite extension, are not comparable,
even if the associated primes are. See \cite{C-E}. Replacing the
fields of norms by $F=(\widehat{L^{ab}})^{\flat}$ is essential for
the inclusion $\widetilde{\mathrm{\mathbf{A}}}_{L_{1}}\subset\widetilde{\mathrm{\mathbf{A}}}_{L_{2}}$,
without which one can not proceed. 

\medskip{}

Besides compatibility with induction and restriction, working with
$\widetilde{\mathrm{\mathbf{A}}}_{L}$ as coefficients instead of
$\mathbf{A}_{L}$ (and the full abelian extension instead of the cyclotomic
or Lubin-Tate tower) has another advantage over the Kisin-Ren modules
studied in \cite{=00005BKi-Re=00005D} and \cite{=00005BSchn=00005D}.
As observed in \cite{=00005BFou-Xie=00005D} and \cite{=00005BBe-Fou=00005D},
the Cherbonnier-Colmez theorem, asserting that étale $(\varphi,\Gamma)$-modules
are overconvergent \cite{=00005BCh-Co=00005D}, no longer holds for
$\mathbf{A}_{L}$-modules in the Lubin-Tate setting, as soon as $L\ne\mathbb{Q}_{p}.$
This has to do with the fact that now $\Gamma\simeq\mathcal{O}_{L}^{\times}$
rather than $\mathcal{\mathbb{Z}}_{p}^{\times}$, and is related to
the question of $L$-analyticity. See the discussion in \cite{=00005BSchn=00005D}
§4.3. In our setting, the Cherbonnier-Colmez theorem \emph{does hold}. 

In §4 we introduce the ring of overconvergent periods $\mathbf{\widetilde{A}}_{L}^{\dagger}$
and the category $\mathrm{Mod}_{\varphi,\Gamma}^{\dagger}(L)$ of
overconvergent $(\varphi,\Gamma)$-modules. Just as $\mathbf{A}_{L}^{\dagger}$
can be realized as a ring of power series converging in some annulus
$R<|X|<1,$ $\mathbf{\widetilde{A}}_{L}^{\dagger}$ can be realized
as a ring of power series converging in some \emph{(pre)perfectoid
}annulus. Base change from $\mathbf{\widetilde{A}}_{L}^{\dagger}$
to $\widetilde{\mathrm{\mathbf{A}}}_{L}$ induces an equivalence of
categories between the full subcategory $\mathrm{Mod}_{\varphi,\Gamma}^{\dagger\acute{e}t}(L)$
of overconvergent étale $(\varphi,\Gamma)$-modules, and $\mathrm{Mod}_{\varphi,\Gamma}^{\acute{e}t}(L)$.
In §4 we give an elementary, self-contained proof of this result.
Once again we are motivated by Kedlaya's paper \cite{=00005BKe=00005D},
but the reader familiar with his proof will notice that we prove the
basic lemmas directly over $L$, rather than after a base change $L'/L$
trivializing the $(\varphi,\Gamma)$-module modulo $p,$ thus avoiding
the need for Galois descent. As explained by Kedlaya, if the field
of norms is replaced by its completed perfection and the ring $\mathbf{A}_{L}$
by $\widetilde{\mathrm{\mathbf{A}}}_{L}$, then overconvergence can
be deduced solely from an analysis of the matrix representing $\varphi$
(at least if there is no $p$-torsion, but torsion modules are easy
to deal with). The structure of $\Gamma$ intervenes only at the second
stage, when one tries to descend to $\mathbf{A}_{L}$ and its overconvergent
subring. It is this second stage that works for the cyclotomic $\Gamma$
but fails in the Lubin-Tate setting if $L\ne\mathbb{Q}_{p}.$

\medskip{}

As should be clear from the introduction, our improvement over what
is already in the literature is modest, and this paper replicates
ideas and results scattered throughout many of the sources that we
have cited. Nevertheless, the better compatibility with induction
and restriction, as well as with overconvergence, makes one wonder
whether $\mathbf{\widetilde{A}}_{L}$ should not substitute for $\mathbf{A}_{L}$
as the basic period ring. Development of $p$-adic Hodge theory so
far relied on descent from $\mathbf{\widetilde{A}}_{L}$ to $\mathbf{A}_{L}$.
For example, the very definition of the operator $\psi$ relies on
$\varphi$ \emph{not} being bijective, and the study of locally analytic
vectors and $p$-adic differential equations is also conducted over
the \emph{non-perfectoid} Robba ring. However, the predominance of
perfectoid rings in Scholze's work and their many applications, together
with the two observations made above, support such thoughts.

\section{$(\varphi,\Gamma)$-modules over $L$}

\subsection{Tilting}

The following fundamental constructions are due to Fontaine and Scholze.
Let $L$ be a finite extension of $\mathbb{Q}_{p}$ and write $\mathbb{C}_{p}$
for the completion of a fixed algebraic closure $\overline{L}$ of
$L$. If
\[
L\subset K\subset\mathbb{C}_{p}
\]
is any complete intermediate field, we write $K^{\flat}$ for the
collection of sequences
\[
x=(\dots,x_{2},x_{1},x_{0})
\]
where $x_{i}\in K$ and $x_{i+1}^{p}=x_{i}$. We define $x\cdot y$
by component-wise multiplication and $x+y=z$ where
\[
z_{i}=\lim_{j\to\infty}(x_{i+j}+y_{i+j})^{p^{j}}
\]
(the limit exists). We let $|x|_{\flat}=|x_{0}|.$ Then $K^{\flat}$
becomes a field of characteristic $p$ and $|\cdot|_{\flat}$ is a
complete non-archimedean absolute value on $K^{\flat}.$ For example,
$L^{\flat}\simeq\mathbb{F}_{q}$, the residue field of $\mathcal{O}=\mathcal{O}_{L}.$
The field $K^{\flat}\subset\mathbb{C}_{p}^{\flat}$ is called the
\emph{tilt} of $K$.

Recall that $K$ is called a \emph{perfectoid} \cite{=00005BScho=00005D}
if it is in addition non-discrete and every element of the ring $\mathcal{O}_{K}/p\mathcal{O}_{K}$
is a $p$-th power. In this case $K^{\flat}$ is perfect, and can
be identified with the field of fractions of the perfection of $\mathcal{O}_{K}/p\mathcal{O}_{K}$.
The field $\mathbb{C}_{p}$ itself is such a perfectoid.

The group $Aut_{cont}(K/L)$ of continuous automorphisms of $K$ over
$L$ acts by functoriality on $K^{\flat}$.

As an example, consider a Lubin-Tate formal group $\mathcal{G}$ over
$\mathcal{O}$, associated with the prime $\pi$ of $\mathcal{O}$.
We fix a formal parameter $X$ on $\mathcal{G}$ and denote by $[a]$
the endomorphism of $\mathcal{G}$ whose expression in $X$ starts
with $aX+(\mathrm{higher\,terms})$. We let $\phi=[\pi],$ so that
\[
\phi(X)=\pi X+\cdots\equiv X^{q}\mod\pi.
\]

For $n\ge1$ let $\omega_{n}\in\mathcal{G}[\pi^{n}]$ be such that
$\omega_{1}\ne0$ and $[\pi](\omega_{n+1})=\omega_{n}.$ Let $L_{n}=L(\omega_{n})$
and $L_{\infty}=\bigcup L_{n}.$ Its completion $K=\widehat{L_{\infty}}$
is a perfectoid subfield of $\mathbb{C}_{p}$. Since $\omega_{n+1}^{q}\equiv\omega_{n}\mod\pi$,
the choice of $\omega_{n}$ defines an element $\omega\in K^{\flat}$
and $\kappa((\omega))\subset K^{\flat}$ where $\mathbb{\kappa=\mathcal{O}}/\pi\mathcal{O}\simeq\mathbb{F}_{q}.$
To be precise, letting $q=p^{f},$ $\omega$ is the unique element
$x\in K^{\flat}$ in which $x_{nf}\equiv\omega_{n+m}^{q^{m}}\mod p$
for large enough $m$. Moreover (\cite{=00005BSchn=00005D}, Prop.
1.4.17), $K^{\flat}$ is the completion of the perfection of $\kappa((\omega)).$
Thus elements of $K^{\flat}$ are formal power series $\sum a_{m}\omega^{m}$
where $m\in\mathbb{Z}[p^{-1}],$ $a_{m}\in\kappa$ and for any real
number $M$ there are only finitely many $m\in\mathbb{Z}[p^{-1}]$
with $m<M$ and $a_{m}\ne0.$ The field $\kappa((\omega))$ is called
the \emph{field of norms} of the extension $L_{\infty}/L,$ and is
independent of the choice of $\omega$. Note that
\[
|\omega|_{\flat}=|\pi|^{q/(q-1)}.
\]

\bigskip{}

In this work, however, we consider $L^{ab}=L^{nr}L_{\infty},$ the
maximal abelian extension of $L$ in $\overline{L}$, and let
\[
K=\widehat{L^{ab}},\,\,\,\,\,\,F=K^{\flat}.
\]
\begin{prop}
(i) The field $K$ is a perfectoid.

(ii) The field $F$ can be identified with the field of formal power
series $\sum a_{m}\omega^{m}$ where $m\in\mathbb{Z}[p^{-1}],$ $a_{m}\in\overline{\kappa}$,
and for any real number $M$ there are only finitely many $m\in\mathbb{Z}[p^{-1}]$
with $m<M$ and $a_{m}\ne0.$ Alternatively, it is the completed perfection
of $\overline{\kappa}((\omega)).$
\end{prop}

\begin{proof}
(i) $K$ is complete and non-discrete. The ring $\mathcal{O}_{K}/p\mathcal{O}_{K}$
is the union of $\mathcal{O}_{L_{d}^{nr}L_{\infty}}/p\mathcal{O}_{L_{d}^{nr}L_{\infty}}$for
$d=1,2,\dots$, where $L_{d}^{nr}$ is the unramified extension of
$L$ of degree $d$. The extension $L_{d}^{nr}L_{\infty}/L$ is arithmetically
profinite (its ramification groups, in the upper numbering, are open),
hence its completion is a perfectoid field, and every element of $\mathcal{O}_{L_{d}^{nr}L_{\infty}}/p\mathcal{O}_{L_{d}^{nr}L_{\infty}}$
is a $p$-th power. It follows that every element of $\mathcal{O}_{K}/p\mathcal{O}_{K}$
is a $p$-th power as well.

(ii) The proof in \cite{=00005BSchn=00005D}, Proposition 1.4.17,
carries over to our case with minor modifications.
\end{proof}
The field $F$ is a non-discrete complete valuation field. Its ring
of integers $\mathcal{O}_{F}$ is the perfection of $\mathcal{O}_{K}/p\mathcal{O}_{K}.$
For $M>0$ we denote by $\mathfrak{m}_{F}^{(M)}$ the ideal of $\mathcal{O}_{F}$
consisting of all $x$ with
\[
|x|_{\flat}\le|\pi|^{Mq/(q-1)}.
\]
 Thus $\mathfrak{m}_{F}^{(M)}$ can be identified with the ideal of
all the formal power series $\sum a_{m}\omega^{m}$ as above with
$a_{m}=0$ for $m<M.$

\medskip{}

The group $\Gamma=Gal(L^{ab}/L)=Aut_{cont}(K/L)$ is isomorphic to
the profinite completion $\widehat{L^{\times}}$ of $L^{\times}.$
We denote by $[a,L^{ab}/L]$ the local Artin symbol of $a\in\widehat{L^{\times}}$
and by $\chi_{L}:\Gamma\simeq\widehat{L^{\times}}$ the character
defined by
\[
\chi_{L}(\gamma)=a\Leftrightarrow\gamma=[a^{-1},L^{ab}/L].
\]
Thus, if $\gamma$ is the geometric Frobenius of $L^{ab}/L_{\infty}$,
$\chi_{L}(\gamma)=\pi.$ Lubin-Tate theory tells us, on the other
hand \cite{=00005BIw=00005D}, that if $\gamma\in Gal(L^{ab}/L^{nr})$,
then $\chi_{L}(\gamma)\in\mathcal{O}_{L}^{\times}$ and
\[
\gamma(\omega_{n})=[\chi_{L}(\gamma)](\omega_{n}).
\]

Let $H=Gal(\overline{L}/L^{ab})=Aut_{cont}(\mathbb{C}_{p}/K)$ and
$G=Gal(\overline{L}/L)=Aut_{cont}(\mathbb{C}_{p}/L)$, so that $G/H=\Gamma.$
Let $F^{sep}$ denote the separable closure of $F$ in $\mathbb{C}_{p}^{\flat}.$
The following Proposition is well-known, see \cite{=00005BScho=00005D}
Theorem 3.7.
\begin{prop}
\label{prop:Perfectoid_basics}(i) The field $F^{sep}$ is algebraically
closed and dense in $\mathbb{C}_{p}^{\flat}.$ 

(ii) The action of $G$ on $\mathbb{C}_{p}^{\flat}$ induces an isomorphism
\[
H\simeq Gal(F^{sep}/F)=Aut_{cont}(\mathbb{C}_{p}^{\flat}/F).
\]

(iii) If $K'$ is a finite extension of $K$ in $\mathbb{C}_{p}$
then $K'$ is a perfectoid, and $[K'^{\flat}:F]=[K':K].$ Every finite
extension of $F$ in $\mathbb{C}_{p}^{\flat}$ is $K'^{\flat}$ for
a unique $K'$ as above, and if $K'$ is Galois, $Gal(K'/K)\simeq Gal(K'^{\flat}/F).$

(iv) (Ax-Sen-Tate theorem) $\mathbb{C}_{p}^{G}=L$ and $\mathbb{C}_{p}^{H}=K.$
Likewise $(\mathbb{C}_{p}^{\flat})^{H}=F.$
\end{prop}

If $\alpha\in F$ we shall sometimes write $|\alpha|$ for $|\alpha|_{\flat}$,
to ease the notation. We also write $\varphi(\alpha)=\alpha^{q}$
for the Frobenius automorphism of order $q$. If $\alpha\ne0$
\[
\lim_{n\to\infty}|\varphi^{-n}(\alpha)|=1.
\]

If $A=(\alpha_{ij})\in M_{d}(F)$ we let $|A|=\max\{|\alpha_{ij}|\}.$
Then
\[
|A+B|\le\max\{|A|,|B|\},\,\,\,\,|AB|\le|A||B|.
\]

The following technical Lemma on matrices will be needed in the section
on overconvergence. The reader interested only in the formulas for
induction and restriction of $(\varphi,\Gamma)$-modules, can skip
it.
\begin{lem}
\label{lem: estimate}Let $A\in GL_{d}(F)$ be given. Then there exists
a constant $c$, depending only on $A$, so that for every $B\in M_{d}(F)$
there exist $U,V\in M_{d}(F)$ with $|V|\le c$ such that
\[
A^{-1}\varphi(U)A-U=B-V.
\]
\end{lem}

\begin{proof}
Let
\[
U=\sum_{i=1}^{N}\varphi^{-1}(A)\varphi^{-2}(A)\cdots\varphi^{-i}(A)\cdot\varphi^{-i}(B)\cdot\varphi^{-i}(A)^{-1}\cdots\varphi^{-2}(A)^{-1}\varphi^{-1}(A)^{-1}.
\]
Then
\[
A^{-1}\varphi(U)A-U=B-V
\]
where
\[
V=\varphi^{-1}(A)\varphi^{-2}(A)\cdots\varphi^{-N}(A)\cdot\varphi^{-N}(B)\cdot\varphi^{-N}(A)^{-1}\cdots\varphi^{-2}(A)^{-1}\varphi^{-1}(A)^{-1}.
\]
Now
\[
|\varphi^{-1}(A)\varphi^{-2}(A)\cdots\varphi^{-N}(A)|\le|A|^{q^{-1}+q^{-2}+\cdots+q^{-N}}
\]
is bounded independently of $N,$ and similarly $|\varphi^{-N}(A)^{-1}\cdots\varphi^{-2}(A)^{-1}\varphi^{-1}(A)^{-1}|.$
On the other hand, by selecting $N$ large enough we can let $|\varphi^{-N}(B)|$
be as close as we want to 1. This concludes the proof, in fact with
any $c>(|A||A^{-1}|)^{1/(q-1)}.$
\end{proof}

\subsection{The coefficient ring}

Consider the usual ring of Witt vectors $W(F).$ It contains the subring
$W(\kappa)$, which is the ring of integers of $L^{0}=L\cap\mathbb{Q}_{p}^{nr}.$
We let
\[
\widetilde{\mathbf{A}}_{L}=W(F)_{L}=\mathcal{O}\otimes_{W(\kappa)}W(F).
\]
The action of $\Gamma$ on $F$ defines an action of $\Gamma$ on
$W(F),$ and as it fixes $W(\kappa)$ point-wise, it extends to $\widetilde{\mathbf{A}}_{L}$
$\mathcal{O}$-linearly. Similarly, letting $\varphi(x)=x^{q}$ be
the Frobenius automorphism of order $q$ of $F,$ we denote by $\varphi$
the induced $\mathcal{O}$-linear automorphism of $\widetilde{\mathbf{A}}_{L}.$
The actions of $\Gamma$ and $\varphi$ on $\widetilde{\mathbf{A}}_{L}$
commute with each other. The structure of $\widetilde{\mathbf{A}}_{L}$
is given in the next Proposition. Let $v_{p}$ be the $p$-adic valuation
on $\mathbb{C}_{p}^{\times},$ normalized by $v_{p}(p)=1$, fix a
Lubin-Tate group $\mathcal{G}$ over $\mathcal{O}$ associated with
the prime $\pi$, and let $L_{\infty}/L$ be its Lubin-Tate tower,
$L_{n}=L(\omega_{n})$ as before.
\begin{prop}
\label{prop:Another look at A_L}(i) Let $W_{L}$ be the completion
of the ring of integers in the maximal unramified extension of $L$,
i.e. $W_{L}=\mathcal{O}\otimes_{W(\kappa)}W(\overline{\kappa}).$
The $\mathcal{O}$-algebra $\widetilde{\mathbf{A}}_{L}$ is isomorphic
to the ring of all power series $\sum a_{m}X^{m}$ where $m\in\mathbb{Z}[p^{-1}],$
$a_{m}\in W_{L}$ and for any real number $M$ there are only finitely
many $m$ with $m<M$ and $v_{p}(a_{m})<M.$ Under reduction modulo
$\pi$ (a prime of $\mathcal{O}$) $X^{m}$ goes to $\omega^{m}$.

(ii) Under the isomorphism in (i), if $\gamma\in Gal(L^{ab}/L^{nr})$
then
\[
\gamma(a_{m})=a_{m},\,\,\,\,\gamma(X)=\lim_{n\to\infty}([\chi_{L}(\gamma)](X^{1/q^{n}}))^{q^{n}},
\]
$Gal(L^{ab}/L_{\infty})\simeq Gal(L^{nr}/L)$ acts on $a_{m}$ via
the natural action on $W_{L}$ and trivially on $X,$ while $\varphi$
acts naturally on $W_{L}$ and
\[
\varphi(X)=X^{q}.
\]
\end{prop}

\begin{proof}
(i) Temporarily, let $\mathscr{A}_{L}$ denote the power series ring
in (i). It is readily checked that $\mathscr{A}_{L}$ is a strict
$p$-ring (in the category of $\mathcal{O}$-algebras) with $\mathscr{A}_{L}/\pi\mathscr{A}_{L}\simeq F$,
the isomorphism sending $a\in W_{L}$ to its reduction modulo $\pi$
and $X^{m}$ to $\omega^{m}.$ This suffices to establish the existence
of a unique isomorphism
\[
\iota:\mathscr{A}_{L}\simeq\widetilde{\mathbf{A}}_{L}
\]
compatible with the given isomorphism after reduction modulo $\pi$.

(ii) If $x\in F$ we let $\tau(x)$ be its Teichmüller representative
in $\widetilde{\mathbf{A}}_{L}.$ For $m\in\mathbb{Z}[p^{-1}]$ we
have $X^{m}=\tau(\omega^{m}).$ The formulae in (ii) are then clear. 
\end{proof}
We remark that in \cite{=00005BSchn=00005D} the Teichmüller representative
$\tau(\omega)$ is modified to obtain $\omega_{\phi}\in\mathbf{A}_{L}$
satisfying $\gamma(\omega_{\phi})=[\chi_{L}(\gamma)](\omega_{\phi})$
and $\varphi(\omega_{\phi})=[\pi](\omega_{\phi}).$ The Cohen ring
$\mathbf{A}_{L}$ is then the $p$-adic completion of $W_{L}((\omega_{\phi})).$
This $\omega_{\phi}$ however can not be raised to a power $m\in\mathbb{Z}[p^{-1}].$

\bigskip{}

The discrete valuation ring $\widetilde{\mathbf{A}}_{L}$ has two
topologies. The \emph{strong} topology is the one given by the valuation.
The \emph{weak} topology is induced on $W(F)$ via the natural bijection
with $F^{\mathbb{N}_{0}}$ from the product topology on the latter.
It is then extended naturally to $W(F)_{L}$, which, as an additive
group, is isomorphic to $W(F)^{[L:L^{0}]}.$ A basis of open neighborhoods
at 0 in the weak topology is given by 
\[
U_{n,m}=\pi^{n}W(F)_{L}+W(\mathfrak{m}_{F}^{(m)})_{L}
\]
for $n,m\ge0.$ The weak topology is a complete Hausdorff topology,
but unlike the situation in \cite{=00005BSchn=00005D}, the subring
$W(\mathcal{O}_{F})_{L}$ is not compact. The automorphisms $\varphi,\varphi^{-1}$
and the action $\Gamma\times\widetilde{\mathbf{A}}_{L}\to\widetilde{\mathbf{A}}_{L}$
are continuous in the weak topology of $\widetilde{\mathbf{A}}_{L}$,
but the orbit map $\Gamma\to\widetilde{\mathbf{A}}_{L},$ $\gamma\mapsto\gamma x$
is not continuous, for a general $x\in\widetilde{\mathbf{A}}_{L},$
in the strong topology.

\subsection{$(\varphi,\Gamma)$-modules over $L$}

As in \cite{=00005BSchn=00005D}, Exercise 2.2.3, any finitely generated
$\mathbf{\widetilde{A}}_{L}$-module carries a canonical topology
called the \emph{weak topology}, which may be defined as the quotient
topology of any surjective homomorphism $\mathbf{\widetilde{A}}_{L}^{n}\twoheadrightarrow M$.
We make the following definition.
\begin{defn}
(i) A $(\varphi,\Gamma)$-module over $L$ is a finitely generated
$\mathbf{\widetilde{A}}_{L}$-module $M$ equipped with a $\varphi$-semilinear
endomorphism
\[
\varphi_{M}:M\to M
\]
and a $\Gamma$-semilinear action
\[
\Gamma\times M\to M,
\]
which is bi-continuous (when $\Gamma$ is given its Krull topology
and $M$ its canonical weak topology), and which commutes with $\varphi_{M}$.

(ii) A homomorphism between $(\varphi,\Gamma)$-modules over $L$
is a homomorphism of $\mathbf{\widetilde{A}}_{L}$-modules $\alpha:M\to N$
which commutes with the $\Gamma$-action and satisfies $\alpha\circ\varphi_{M}=\varphi_{N}\circ\alpha$.

(iii) A $(\varphi,\Gamma)$-module $M$ over $L$ is called \emph{étale}
if $\varphi_{M}$ is bijective.
\end{defn}

\emph{Remark about topologies. }(i) A homomorphism $\alpha$ as above,
as well as the semilinear $\varphi_{M}$, are automatically continuous.
The arguments from \cite{=00005BSchn=00005D}, Remark 2.2.5, remain
valid.

(ii) In \cite{=00005BSchn=00005D}, Theorem 2.2.8, it is shown that
if $M$ is an étale $\varphi$-module, every semilinear $\Gamma$-action
which commutes with $\varphi$ is automatically bi-continuous. The
proof of this useful fact relies on the local compactness of the field
of norms. As our $F$ is not locally compact, we do not know if we
can give up the continuity assumption in the definition, even if $M$
is étale.

\medskip{}

We denote by 
\[
\mathrm{Mod}_{\varphi,\Gamma}^{\acute{e}t}(L)\subset\mathrm{Mod}_{\varphi,\Gamma}(L)
\]
the category of $(\varphi,\Gamma)$-modules over $L$, and its full
subcategory of étale $(\varphi,\Gamma)$-modules. These are abelian
categories and the forgetful functors from them to the category of
$\mathbf{\widetilde{A}}_{L}$-modules are exact. 

\section{Equivalence of categories}

\subsection{The functors}

Let
\[
\mathbf{\widetilde{A}}=\mathcal{O}\otimes_{W(\kappa)}W(F^{sep}),
\]
and let $\varphi$ continue to denote the $q$-power Frobenius of
$W(F^{sep}),$ extended linearly to $\mathbf{\widetilde{A}}$. The
Galois group $G$ acts on $F^{sep}$, hence on $\mathbf{\widetilde{A}}$.
The weak topology of $\mathbf{\widetilde{A}}$ is defined as it was
defined for $\mathbf{\widetilde{A}}_{L}.$ As before, $\varphi$ and
the $G$-action on $\mathbf{\widetilde{A}}$ are continuous for the
weak topology, and 
\begin{equation}
\mathbf{\widetilde{A}}^{H}=\mathbf{\widetilde{A}}_{L},\,\,\,\mathbf{\widetilde{A}}^{\varphi}=\mathcal{O}.\label{eq:invariants}
\end{equation}

Given $V\in\mathrm{Rep}{}_{\mathcal{O}}(L)$, let
\[
\mathcal{D}(V)=(\mathbf{\widetilde{A}}\otimes_{\mathcal{O}}V)^{H},
\]
where the fixed points of $H$ are taken with respect to the diagonal
action. By $(\ref{eq:invariants}),$ $\mathcal{D}(V)$ is an $\mathbf{\widetilde{A}}_{L}$-module.
The diagonal action of $G$ yields a residual semilinear $\Gamma$-action
on $\mathcal{D}(V),$ and since $\varphi\otimes1$ and $H$ commute
in their action on $\mathbf{\widetilde{A}}\otimes_{\mathcal{O}}V$,
$\varphi\otimes1$ induces a semilinear endomorphism $\varphi_{\mathcal{D}(V)}$
on $\mathcal{D}(V).$
\begin{lem}
With the above definitions, $\mathcal{D}(V)\in\mathrm{Mod}_{\varphi,\Gamma}^{\acute{e}t}(L)$.
\end{lem}

\begin{proof}
The proof that $\mathcal{D}(V)$ is an \emph{étale} $\varphi$-module
is straightforward, since $\varphi$ is bijective on $\mathbf{\widetilde{A}}_{L}$.
The key step is the proof that $\mathcal{D}(V)$ is a finitely generated
module over the discrete valuation ring $\mathbf{\widetilde{A}}_{L},$
and that the homomorphism
\[
ad_{V}:\mathbf{\widetilde{A}}\otimes_{\mathbf{\widetilde{A}}_{L}}\mathcal{D}(V)\to\mathbf{\widetilde{A}}\otimes_{\mathcal{O}}V,\,\,\,\,a\otimes m\mapsto am
\]
is bijective. This is done first under the assumption that $V$ is
killed by $\pi$, with the help of Hilbert's theorem 90, then by dévissage
for torsion $V$'s, and finally, taking inverse limits, for general
$V$. For the details see \cite{=00005BSchn=00005D}. These two facts
also imply (see \cite{=00005BSchn=00005D}, Lemma 3.1.10) that the
$G$-action on $\mathbf{\widetilde{A}}\otimes_{\mathcal{O}}V$ is
continuous, hence the $\Gamma$-action on $\mathcal{D}(V)$ is continuous.
\end{proof}
We next define a functor in the opposite direction. Let $M\in\mathrm{Mod}_{\varphi,\Gamma}^{\acute{e}t}(L),$
and define
\[
\mathcal{V}(M)=(\mathbf{\widetilde{A}}\otimes_{\mathbf{\widetilde{A}}_{L}}M)^{\varphi\otimes\varphi_{M}}.
\]
By (\ref{eq:invariants}), this is an $\mathcal{O}$-module. Since
the diagonal Galois action of $G$ on $\mathbf{\widetilde{A}}\otimes_{\mathbf{\widetilde{A}}_{L}}M$
commutes with $\varphi\otimes\varphi_{M}$, $\mathcal{V}(M)$ carries
a $G$-action.
\begin{lem}
With the above definition, $\mathcal{V}(M)\in\mathrm{Rep}{}_{\mathcal{O}}(L)$.
\end{lem}

\begin{proof}
Once again, the key is the proof that $\mathcal{V}(M)$ is finitely
generated over $\mathcal{O}$, and that the homomorphism
\[
ad_{M}:\mathbf{\widetilde{A}}\otimes_{\mathcal{O}}\mathcal{V}(M)\to\mathbf{\widetilde{A}}\otimes_{\mathbf{\widetilde{A}}_{L}}M,\,\,\,\,a\otimes v\mapsto av
\]
is bijective. This is done first under the assumption that $M$ is
killed by $\pi$ (i.e. is an $F$-vector space), using \cite{=00005BSchn=00005D}
proposition 3.2.4, then when $M$ is killed by some $\pi^{n}$ by
dévissage, and finally, taking inverse limits, for general $M$. Compare
with \cite{=00005BSchn=00005D}, Proposition 3.3.9.
\end{proof}

\subsection{The equivalence of categories}

The main theorem is the following.
\begin{thm}
\label{thm:Equivalence}The functors $\mathcal{D}$ and $\mathcal{V}$
are equivalences of categories between $\mathrm{Mod}_{\varphi,\Gamma}^{\acute{e}t}(L)$
and $\mathrm{Rep}{}_{\mathcal{O}}(L),$ and are quasi-inverse to each
other.
\end{thm}

\begin{proof}
One first proves that $\mathcal{D}(\mathcal{V}(M))=M$ and $\mathcal{V}(\mathcal{D}(V))=V$
under the assumption that $M$ and $V$ are killed by $\pi.$ Here
the key is that $ad_{V}$ and $ad_{M}$ are both bijective. Next,
one checks that the two functors are exact and commute with inverse
limits, and then one concludes as in \cite{=00005BSchn=00005D}, Theorem
3.3.10.
\end{proof}

\subsection{Elementary divisors}

If $(R,\mathfrak{m})$ is a discrete valuation ring and $X$ is a
finitely generated $R$-module we write
\[
[X:R]=(r;e_{1},\dots,e_{n})
\]
if $r\ge0,$ $e_{1}\ge e_{2}\ge\cdots\ge e_{n}\ge1$ and $X\simeq R^{r}\oplus(R/\mathfrak{m}^{e_{1}})\oplus\cdots\oplus(R/\mathfrak{m}^{e_{n}})$.
The rank $r$ and the elementary divisors $e_{i}$ are uniquely determined
and characterize $X$ up to isomorphism. The following is well-known
and easy.
\begin{lem}
\label{lem:surjectivity lemma}If $[X:R]=[Y:R]$ and $\alpha:X\to Y$
is a surjective homomorphism, then $\alpha$ is an isomorphism.
\end{lem}

\begin{prop}
\label{prop:elementary divisors}Let $M\in\mathrm{Mod}{}_{\varphi,\Gamma}^{\acute{e}t}(L)$
and $V\in\mathrm{Rep}{}_{\mathcal{O}}(L)$ correspond to each other
under $\mathcal{V}$ and $\mathcal{D}$. Then
\[
[M:\mathbf{\widetilde{A}}_{L}]=[V:\mathcal{O}].
\]
\end{prop}

\begin{proof}
We have an obvious string of equalities
\[
[V:\mathcal{O}]=[\mathbf{\widetilde{A}}\otimes_{\mathcal{O}}V:\mathbf{\widetilde{A}}]=[\mathbf{\widetilde{A}}\otimes_{\mathbf{\widetilde{A}}_{L}}M:\mathbf{\widetilde{A}}]=[M:\mathbf{\widetilde{A}}_{L}].
\]
The middle equality stems of course from the fact that $ad_{M}$ and
its inverse $ad_{V}$ are isomorphisms.
\end{proof}

\subsection{\label{subsec:A-remark-on-the-use}A remark on the use of $\mathbf{\widetilde{A}}$}

It is possible to define the functors $\mathcal{D}$ and $\mathcal{V}$
using, instead of $\mathbf{\widetilde{A}}=\mathcal{O}\otimes_{W(\kappa)}W(F^{sep})$,
the larger ring
\[
\mathbf{\widehat{A}}=\mathcal{O}\otimes_{W(\kappa)}W(\mathbb{C}_{p}^{\flat}).
\]
The reason is that for $V\in\mathrm{Rep}{}_{\mathcal{O}}(L)$ and
$M=\mathcal{D}(V)$
\[
(\mathbf{\widehat{A}}\otimes_{\mathcal{O}}V)^{H}=(\mathbf{\widehat{A}}\otimes_{\mathbf{\widetilde{A}}}(\mathbf{\widetilde{A}}\otimes_{\mathcal{O}}V))^{H}=(\mathbf{\widehat{A}}\otimes_{\mathbf{\widetilde{A}}}(\mathbf{\widetilde{A}}\otimes_{\mathbf{\widetilde{A}}_{L}}M))^{H}
\]

\[
=(\mathbf{\widehat{A}}\otimes_{\mathbf{\widetilde{A}}_{L}}M)^{H}=\widehat{\mathbf{A}}^{H}\otimes_{\mathbf{\widetilde{A}}_{L}}M=M
\]
since $(\mathbb{C}_{p}^{\flat})^{H}=F$. A similar argument works
for the functor $\mathcal{V}$.

\section{Restriction and Induction}

\subsection{Definitions of the two functors}

Let $L_{1}\subset L_{2}$ be two finite extensions of $\mathbb{Q}_{p}$
contained in $\mathbb{\overline{Q}}_{p}.$ Let $\mathcal{O}_{1}\subset\mathcal{O}_{2}$
be their rings of integers, $\kappa_{1}\subset\kappa_{2}$ their residue
fields, and let $G_{i},H_{i}$ and $\Gamma_{i}$ be the groups defined
before, with $L_{i}$ as $L$. Write
\[
d=[\kappa_{2}:\kappa_{1}]
\]
for the inertial degree of $L_{2}/L_{1}.$ Letting $\varphi_{i}$
denote the Frobenius automorphism of $L_{i}^{nr}/L_{i}$, we have
that
\[
\varphi_{2}|_{L_{1}^{nr}}=\varphi_{1}^{d}.
\]
Let
\[
\mathrm{Res}{}_{L_{1}}^{L_{2}}:\mathrm{Rep}{}_{\mathcal{O}_{1}}(L_{1})\to\mathrm{Rep}{}_{\mathcal{O}_{2}}(L_{2})
\]
be the functor $\mathrm{Res}{}_{L_{1}}^{L_{2}}(V)=\mathcal{O}_{2}\otimes_{\mathcal{O}_{1}}\mathrm{Res}{}_{G_{1}}^{G_{2}}(V).$
Thus, we restrict the group action to a smaller subgroup and extend
scalars. Similarly let
\[
\mathrm{Ind}{}_{L_{1}}^{L_{2}}:\mathrm{Rep}{}_{\mathcal{O}_{2}}(L_{2})\to\mathrm{Rep}{}_{\mathcal{O}_{1}}(L_{1})
\]
be the functor $\mathrm{Ind}{}_{L_{1}}^{L_{2}}(W)=\mathrm{Res}{}_{\mathcal{O}_{1}}^{\mathcal{O}_{2}}(\mathrm{Ind}{}_{G_{1}}^{G_{2}}(W)).$
In this case, we take the induced module, which is an $\mathcal{O}_{2}$-module
with a $G_{1}$-action, but view it solely as an $\mathcal{O}_{1}$-module.

The two functors are adjoints of each other: there is a functorial
isomorphism
\[
\mathrm{Hom}{}_{\mathrm{Rep}{}_{\mathcal{O}_{1}}(L_{1})}(V,\mathrm{Ind}{}_{L_{1}}^{L_{2}}(W))\simeq\mathrm{Hom}{}_{\mathrm{Rep}{}_{\mathcal{O}_{2}}(L_{2})}(\mathrm{Res}{}_{L_{1}}^{L_{2}}(V),W).
\]
Our goal is to describe the corresponding functors between the categories
$\mathrm{Mod}_{\varphi,\Gamma}^{\acute{e}t}(L_{i})$. We shall construct
functors
\[
\mathcal{R}_{L_{1}}^{L_{2}}:\mathrm{Mod}_{\varphi,\Gamma}(L_{1})\to\mathrm{Mod}_{\varphi,\Gamma}(L_{2}),\,\,\,\,\mathcal{I}_{L_{1}}^{L_{2}}:\mathrm{Mod}_{\varphi,\Gamma}(L_{2})\to\mathrm{Mod}_{\varphi,\Gamma}(L_{1}),
\]
show that they respect the full subcategories of étale $(\varphi,\Gamma)$-modules,
and that the following diagram is commutative

\begin{equation}
\xymatrix{\mathrm{Rep}_{\mathcal{O}_{1}}(L_{1})\ar@<1ex>[d]^{\mathcal{D}_{1}}\ar[r]^{\mathrm{Res}_{L_{1}}^{L_{2}}} & \mathrm{Rep}_{\mathcal{O}_{2}}(L_{2})\ar@<1ex>[d]^{\mathcal{D}_{2}}\ar[r]^{\mathrm{Ind}_{L_{1}}^{L_{2}}} & \mathrm{Rep}_{\mathcal{O}_{1}}(L_{1})\ar@<1ex>[d]^{\mathcal{D}_{1}}\\
\mathrm{Mod}_{\varphi,\Gamma}^{\textrm{ét}}(L_{1})\ar@<1ex>[u]^{\mathcal{V}_{1}}\ar[r]^{\mathcal{R}_{L_{1}}^{L_{2}}} & \mathrm{Mod}_{\varphi,\Gamma}^{\textrm{ét}}(L_{2})\ar@<1ex>[u]^{\mathcal{V}_{2}}\ar[r]^{\mathcal{I}_{L_{1}}^{L_{2}}} & \mathrm{Mod}_{\varphi,\Gamma}^{\textrm{ét}}(L_{1})\ar@<1ex>[u]^{\mathcal{V}_{1}}
}
\label{eq:commutativity}
\end{equation}

Chasing a diagram of functors as above means that we have to check
commutativity both on objects and on morphisms. We shall do it on
objects, leaving out verifications, e.g. that $\mathcal{D}\circ\mathrm{Ind}{}_{L_{1}}^{L_{2}}\circ\mathcal{V}$
agrees with $\mathcal{I}_{L_{1}}^{L_{2}}$ on morphisms, to the reader.

Note first that since $L_{1}\subset L_{2},$ also $L_{1}^{ab}\subset L_{2}^{ab},$
hence $K_{1}\subset K_{2}\subset\mathbb{C}_{p}$ and $F_{1}\subset F_{2}\subset\mathbb{C}_{p}^{\flat}$.
It follows that
\[
\mathbf{\widetilde{A}}_{L_{1}}\subset\mathbf{\widetilde{A}}_{L_{2}},\,\,\,\,\,\mathbf{\widetilde{A}}_{1}\subset\mathbf{\widetilde{A}}_{2}.
\]
Second, there is a natural group homomorphism
\[
r:\Gamma_{2}\to\Gamma_{1}
\]
given by $r(\gamma_{2})=\gamma_{2}|_{L_{1}^{ab}}$. Its image is $Gal(L_{1}^{ab}/L_{2}\cap L_{1}^{ab}),$
which is of finite index in $\Gamma_{1}$. Its kernel
\[
\Gamma_{12}=\ker(r)
\]
is the subgroup $Gal(L_{2}^{ab}/L_{1}^{ab}L_{2}).$ If we let $H_{12}=Gal(\overline{L}_{1}/L_{1}^{ab}L_{2})$
then this kernel is identified with $H_{12}/H_{2}.$ Via the local
Artin maps, $r$ is identified with the norm map $N_{L_{2}/L_{1}}:\widehat{L_{2}^{\times}}\to\widehat{L_{1}^{\times}}.$
The inclusion $\mathbf{\widetilde{A}}_{L_{1}}\subset\mathbf{\widetilde{A}}_{L_{2}}$
is compatible with the homomorphism $r$.

To define $\mathcal{R}_{L_{1}}^{L_{2}}$ let $M_{1}\in\mathrm{Mod}{}_{\varphi,\Gamma}(L_{1}).$
Then simply put
\[
M_{2}=\mathcal{R}_{L_{1}}^{L_{2}}(M_{1})=\mathbf{\widetilde{A}}_{L_{2}}\otimes_{\mathbf{\widetilde{A}}_{L_{1}}}M_{1}
\]
with the $\Gamma_{2}$-action
\[
\gamma_{2}(\lambda\otimes m)=\gamma_{2}(\lambda)\otimes r(\gamma_{2})(m)
\]
and $\varphi_{M_{2}}$ given by
\[
\varphi_{M_{2}}(\lambda\otimes m)=\varphi_{2}(\lambda)\otimes\varphi_{M_{1}}^{d}(m).
\]
One easily checks that $M_{2}\in\mathrm{Mod}_{\varphi,\Gamma}(L_{2}).$

\bigskip{}

The definition of $\mathcal{I}_{L_{1}}^{L_{2}}$ is a little more
subtle, as this is a case of ``semilinear induction'' not too common
in the literature. Let $\Phi$ be a variable and let $\Phi^{d}$ act
on $M_{2}\in\mathrm{Mod}_{\varphi,\Gamma}(L_{2})$ as $\varphi_{M_{2}}.$
We set
\[
M_{1}=\mathcal{I}_{L_{1}}^{L_{2}}(M_{2})=\left\{ f:\Gamma_{1}\to\mathcal{O}_{1}[\Phi]\otimes_{\mathcal{O}_{1}[\Phi^{d}]}M_{2}|\,f(r(\gamma_{2})\gamma)=(1\otimes\gamma_{2})f(\gamma)\right\} 
\]
(for all $\gamma\in\Gamma_{1},$ $\gamma_{2}\in\Gamma_{2}).$ Note
that for any $f\in M_{1}$, its image lies in $\mathcal{O}_{1}[\Phi]\otimes_{\mathcal{O}_{1}[\Phi^{d}]}M_{2}^{\Gamma_{12}}.$
The structure of $M_{1}$ as an $\mathbf{\widetilde{A}}_{L_{1}}$-module
is given as follows. For $\lambda\in\mathbf{\widetilde{A}}_{L_{1}}$
and $f(\gamma)=\sum_{i=0}^{d-1}\Phi^{i}\otimes m_{i}(\gamma)$ ($m_{i}(\gamma)\in M_{2})$
we put
\[
(\lambda f)(\gamma)=\sum_{i=0}^{d-1}\Phi^{i}\otimes(\varphi_{1}^{-i}\circ\gamma)(\lambda)\cdot m_{i}(\gamma).
\]
The $\Gamma_{1}$-action is given as usual by right translation
\[
(\gamma_{1}f)(\gamma)=f(\gamma\gamma_{1}).
\]
Finally, $\varphi_{M_{1}}$ is given by
\[
(\varphi_{M_{1}}f)(\gamma)=(\Phi\otimes1)f(\gamma).
\]
\begin{prop}
$M_{1}\in\mathrm{Mod}_{\varphi,\Gamma}(L_{1}).$
\end{prop}

Everything is easy to check, except that $M_{1}$ is finitely generated
over $\mathbf{\widetilde{A}}_{L_{1}}$. Since $r(\Gamma_{2})$ is
of finite index in $\Gamma_{1}$, it is enough to show that $M_{2}^{\Gamma_{12}}$
is a finitely generated $\mathbf{\widetilde{A}}_{L_{1}}$-module.
We prove this fact in a sequence of lemmas.

By Proposition \ref{prop:Perfectoid_basics} the field $K_{1}L_{2}$
is a perfectoid,
\[
F_{12}=(K_{1}L_{2})^{\flat}
\]
is a finite extension of $F_{1}$, the inclusion $F_{1}^{sep}\cap F_{2}\subset F_{2}$
is dense,
\[
\Gamma_{12}\simeq Gal(F_{1}^{sep}\cap F_{2}/F_{12})=Aut_{cont}(F_{2}/F_{12}),
\]
and $F_{2}^{\Gamma_{12}}=F_{12}$. For simplicity write $F=F_{12},$
$E=F_{2}$ and $\Gamma=\Gamma_{12}.$
\begin{lem}
\label{lem:Finite dimension}Let $N$ be a finite dimensional vector
space over $E$ equipped with a semilinear action of $\Gamma.$ Put
$N_{0}=N^{\Gamma}.$ Then
\[
E\otimes_{F}N_{0}\simeq N.
\]
In particular, $\dim_{F}N_{0}<\infty.$
\end{lem}

\begin{proof}
Showing that
\[
E\otimes_{F}N_{0}\to N,\,\,\,\,\,a\otimes m\mapsto am
\]
is injective is standard: By way of contradiction, assume $\sum_{i=1}^{t}a_{i}m_{i}=0$
where $m_{i}\in N_{0},$ $a_{i}\in E$ and $t$ is minimal. We may
assume that $a_{1}=1.$ Applying $\gamma\in\Gamma$ to the relation
and subtracting we get a shorter relation, contradicting the minimality
of $t,$ unless all $a_{i}\in E^{\Gamma}=F.$ But this means that
$t=1$ so $m_{1}=0.$

The surjectivity is equivalent to the statement that $H_{cont}^{1}(\Gamma,GL_{n}(E))=0$
(where $n=\dim_{E}N$). The proof of this is similar to the proof
of Proposition 4 in \cite{Sen}. The role of \emph{loc. cit.} Proposition
1 is played by the almost-étaleness of the extension $\mathcal{O}_{F'}/\mathcal{O}_{F}$
($[F':F]<\infty$), \emph{cf. }Proposition 5.23 in \cite{=00005BScho=00005D}.
This is the only place where the assumption that $F$ is a perfectoid
is used.
\end{proof}
\begin{lem}
Let $N$ be as above. Then $H_{cont}^{1}(\Gamma,N)=0.$
\end{lem}

\begin{proof}
The previous Lemma reduces this to the statement that $H_{cont}^{1}(\Gamma,E)=0,$
whose proof is again similar to that of the analogous statement in
classical Sen-Tate theory.
\end{proof}
\begin{lem}
$\pi_{1}M_{2}^{\Gamma_{12}}=(\pi_{1}M_{2})^{\Gamma_{12}}.$
\end{lem}

\begin{proof}
Consider the long exact sequence in cohomology associated to the short
exact sequence
\[
0\to M_{2}[\pi_{1}]\to M_{2}\overset{\pi_{1}}{\to}\pi_{1}M_{2}\to0.
\]
It is enough to show that $H_{cont}^{1}(\Gamma_{12},M_{2}[\pi_{1}])=0.$
Noting that $M_{2}[\pi_{1}]=M_{2}[\pi_{2}^{e}]$ where $e$ is the
index of ramification of $L_{2}/L_{1},$ this last fact follows by
dévissage from the previous Lemma.
\end{proof}
We can now conclude the proof of the Proposition. No non-zero element
of $M_{2}^{\Gamma_{12}}$ is divisible by $\pi_{1}^{n}$ for all $n$,
because the same is true in $M_{2}$, which is finitely generated
over the DVR $\mathbf{\widetilde{A}}_{L_{2}}$. By a well-known version
of Nakayama's Lemma it is enough to prove that $M_{2}^{\Gamma_{12}}/\pi_{1}M_{2}^{\Gamma_{12}}$
is finite dimensional over $\mathbf{\widetilde{A}}_{L_{1}}/\pi_{1}\mathbf{\widetilde{A}}_{L_{1}}=F_{1}.$
By the last Lemma,
\[
M_{2}^{\Gamma_{12}}/\pi_{1}M_{2}^{\Gamma_{12}}=M_{2}^{\Gamma_{12}}/(\pi_{1}M_{2})^{\Gamma_{12}}\hookrightarrow(M_{2}/\pi_{1}M_{2})^{\Gamma_{12}},
\]
so it is enough to prove that $(M_{2}/\pi_{1}M_{2})^{\Gamma_{12}}$
is finite dimensional over $F_{1}.$ By dévissage, it is enough to
prove that $(M_{2}/\pi_{2}M_{2})^{\Gamma_{12}}$ is finite dimensional
over $F_{12}.$ This was established in Lemma \ref{lem:Finite dimension}.

\bigskip{}

With the Proposition being settled, we have checked that the two functors
$\mathcal{R}_{L_{1}}^{L_{2}}$ and $\mathcal{I}_{L_{1}}^{L_{2}}$
are well-defined on objects. Their definition on morphisms is self-evident
and is left to the reader.

\subsection{The main theorem}

It is the following.
\begin{thm}
\label{thm:Ind-Res}The two functors $\mathcal{R}_{L_{1}}^{L_{2}}$
and $\mathcal{I}_{L_{1}}^{L_{2}}$ respect the full subcategories
of étale $(\varphi,\Gamma)$-modules, and the diagram (\ref{eq:commutativity})
commutes.
\end{thm}

\begin{proof}
We shall show that for $M_{1}\in\mathrm{Mod}_{\varphi,\Gamma}^{\acute{e}t}(L_{1})$
\begin{equation}
\mathcal{R}_{L_{1}}^{L_{2}}(M_{1})=\mathcal{D}_{2}\circ Res_{L_{1}}^{L_{2}}\circ\mathcal{V}_{1}(M_{1}),\label{eq:Restriction}
\end{equation}
and that for $M_{2}\in\mathrm{Mod}_{\varphi,\Gamma}^{\acute{e}t}(L_{2})$
\begin{equation}
\mathcal{I}_{L_{1}}^{L_{2}}(M_{2})=\mathcal{D}_{1}\circ Ind_{L_{1}}^{L_{2}}\circ\mathcal{V}_{2}(M_{2}).\label{eq:Induction}
\end{equation}
This will imply both that the functors respect étale $(\varphi,\Gamma)$-modules,
and that the diagram commutes on objects. As mentioned above, we leave
to the reader to check that it commutes on morphisms as well.

\bigskip{}

We start with $\mathcal{R}_{L_{1}}^{L_{2}}$ and note that (\ref{eq:Restriction})
is equivalent to the statement that for $V_{1}=\mathcal{V}_{1}(M_{1})\in\mathrm{Rep}_{\mathcal{O}_{1}}(L_{1})$
\[
\left(\mathbf{\widetilde{A}}_{2}\otimes_{\mathcal{O}_{2}}(\mathcal{O}_{2}\otimes_{\mathcal{O}_{1}}V_{1})\right)^{H_{2}}=\mathbf{\widetilde{A}}_{L_{2}}\otimes_{\mathbf{\widetilde{A}}_{L_{1}}}(\mathbf{\widetilde{A}}_{1}\otimes_{\mathcal{O}_{1}}V_{1})^{H_{1}},
\]
as submodules of $\mathbf{\widetilde{A}}_{2}\otimes_{\mathcal{O}_{1}}V_{1}.$
Thus we have to show that the natural map
\[
\alpha:\mathbf{\widetilde{A}}_{L_{2}}\otimes_{\mathbf{\widetilde{A}}_{L_{1}}}(\mathbf{\widetilde{A}}_{1}\otimes_{\mathcal{O}_{1}}V_{1})^{H_{1}}\to\left(\mathbf{\widetilde{A}}_{2}\otimes_{\mathcal{O}_{1}}V_{1}\right)^{H_{2}}
\]
is bijective. Let $[V_{1}:\mathcal{O}_{1}]=(r;e_{1},\dots,e_{n})$,
and let $X$ be either the source or the target of $\alpha$. Then
Proposition \ref{prop:elementary divisors} implies that
\[
[X:\mathbf{\widetilde{A}}_{L_{2}}]=(r;ee_{1},\dots,ee_{n}),
\]
where $e$ is the ramification index of $L_{2}/L_{1}.$ By Lemma \ref{lem:surjectivity lemma}
it is enough to prove that $\alpha$ is surjective, and for that it
is enough to prove, by Nakayama's Lemma, that it is surjective modulo
$\pi_{2}$. We may therefore assume that $V_{1}$ is a $\kappa_{1}$-vector
space representation of $G_{1}$ and show that
\[
\alpha:F_{2}\otimes_{F_{1}}(F_{1}^{sep}\otimes_{\kappa_{1}}V_{1})^{H_{1}}\to(F_{2}^{sep}\otimes_{\kappa_{1}}V_{1})^{H_{2}}
\]
is an isomorphism of $F_{2}$-vector spaces. Pick a basis $w_{1},\dots,w_{t}$
of $F_{1}^{sep}\otimes_{\kappa_{1}}V_{1}$ over $F_{1}^{sep}$ which
is fixed by $H_{1}$. The vectors $1\otimes w_{i}$ then form a basis
of the left hand side over $F_{2}$ and are mapped by $\alpha$ to
a basis of the right hand side over $F_{2}$. This concludes the proof
of (\ref{eq:Restriction}).

\bigskip{}

We next show (\ref{eq:Induction}) by a direct computation. An alternative
approach, which works equally well, and which we do not pursue, is
to prove the adjointness of $\mathcal{I}_{L_{1}}^{L_{2}}$ and $\mathcal{R}_{L_{1}}^{L_{2}}.$
Below, it will be convenient to make use of the remark from §\ref{subsec:A-remark-on-the-use},
replacing $\widetilde{\mathbf{A}}_{i}$ with $\widehat{\mathbf{A}}_{i}.$
Start with $M_{2}\in\mathrm{Mod}_{\varphi,\Gamma}^{\acute{e}t}(L_{2}),$
$V_{2}=\mathcal{V}_{2}(M_{2})$ as usual. Then $\mathcal{I}_{L_{1}}^{L_{2}}(M_{2})=\mathcal{D}_{1}\circ Ind_{L_{1}}^{L_{2}}\circ\mathcal{V}_{2}(M_{2})$
is
\[
\left\{ \widehat{\mathbf{A}}_{1}\otimes_{\mathcal{O}_{1}}\mathrm{Fun}_{G_{2}}(G_{1},(\widehat{\mathbf{A}}_{2}\otimes_{\mathbf{\widetilde{A}}_{L_{2}}}M_{2})^{\varphi^{d}\otimes\varphi_{2}})\right\} ^{H_{1}}
\]
where $\mathrm{Ind}_{L_{1}}^{L_{2}}V_{2}=\mathrm{Fun}_{G_{2}}(G_{1},V_{2})$
is the space of functions $f:G_{1}\to V_{2}$ satisfying $f(\gamma_{2}\gamma)=\gamma_{2}f(\gamma)$
for all $\gamma_{2}\in G_{2}.$ We regard this space merely as an
$\mathcal{O}_{1}$-module and let $G_{1}$ act on it by right translation:
$(\gamma_{1}f)(\gamma)=f(\gamma\gamma_{1}).$ Since $[G_{1}:G_{2}]<\infty$
the above module is the same as
\[
\mathrm{Fun}_{G_{2}}(G_{1},\widehat{\mathbf{A}}_{1}\otimes_{\mathcal{O}_{1}}\widehat{\mathbf{A}}_{2}\otimes_{\mathbf{\widetilde{A}}_{L_{2}}}M_{2})^{1\otimes\varphi^{d}\otimes\varphi_{2},H_{1}}.
\]
Here the action of $G_{2}$ on the triple tensor product is via the
last two factors only. More precisely, $\gamma_{2}\in G_{2}$ acts
on it via $1\otimes\gamma_{2}\otimes\overline{\gamma}_{2}$ where
$\overline{\gamma}_{2}$ is its image in $\Gamma_{2}=G_{2}/H_{2}.$
The action of $\gamma_{1}\in H_{1}$ is via right translation on $G_{1}$
and via $\gamma_{1}\otimes1\otimes1$ on the triple tensor product,
namely $(\gamma_{1}f)(\gamma)=(\gamma_{1}\otimes1\otimes1)(f(\gamma\gamma_{1})).$

For $f$ in the last space of functions define $f^{\sharp}$ by
\[
f^{\sharp}(\gamma)=(\gamma\otimes1\otimes1)(f(\gamma))
\]
($\gamma\in G_{1})$. Note that $f\mapsto f^{\sharp}$ is \emph{not}
$\mathbf{\widetilde{A}}_{L_{1}}$ -linear. However, the $H_{1}$-invariance
of $f$ is translated to the invariance of $f^{\sharp}$ under right
translation by $H_{1}$, so we can regard $f^{\sharp}$ as a function
on $G_{1}/H_{1}=\Gamma_{1}.$ In addition, the condition $f(\gamma_{2}\gamma)=(1\otimes\gamma_{2}\otimes\overline{\gamma}_{2})(f(\gamma))$
($\gamma_{2}\in G_{2})$ gets translated to the condition
\[
f^{\sharp}(\gamma_{2}\gamma)=(\gamma_{2}\gamma\otimes1\otimes1)(f(\gamma_{2}\gamma))=(\gamma_{2}\gamma\otimes\gamma_{2}\otimes\overline{\gamma}_{2})(f(\gamma))=(\gamma_{2}\otimes\gamma_{2}\otimes\overline{\gamma}_{2})(f^{\sharp}(\gamma)).
\]
Now if $\gamma_{2}\in H_{2}$, it also lies in $H_{1}$, which is
normalized by $G_{1},$ so $f^{\sharp}(\gamma_{2}\gamma)=f^{\sharp}(\gamma\gamma^{-1}\gamma_{2}\gamma)=f^{\sharp}(\gamma)$
by the right-invariance of $f^{\sharp}$ under $H_{1}$. We conclude
that
\[
f^{\sharp}(\gamma)\in(\widehat{\mathbf{A}}_{1}\otimes_{\mathcal{O}_{1}}\widehat{\mathbf{A}}_{2}\otimes_{\mathbf{\widetilde{A}}_{L_{2}}}M_{2})^{1\otimes\varphi^{d}\otimes\varphi_{2},H_{2}}
\]
where the action of $H_{2}$ this time is diagonal, on all three factors.
The group on the right is nothing but $(\widehat{\mathbf{A}}_{1}\otimes_{\mathcal{O}_{1}}V_{2})^{H_{2}}$.

We have reached the following description. With $V_{2}=\mathcal{V}_{2}(M_{2}),$
$\mathcal{I}_{L_{1}}^{L_{2}}(M_{2})$ may be identified with the space
of functions $f^{\sharp}:\Gamma_{1}\to(\widehat{\mathbf{A}}_{1}\otimes_{\mathcal{O}_{1}}V_{2})^{H_{2}}$
satisfying $f^{\sharp}(r(\gamma_{2})\gamma)=\gamma_{2}f^{\sharp}(\gamma)$
for every $\gamma_{2}\in\Gamma_{2}$ and $\gamma\in\Gamma_{1}.$ It
remains to identify the group $(\widehat{\mathbf{A}}_{1}\otimes_{\mathcal{O}_{1}}V_{2})^{H_{2}}$
with $\mathcal{O}_{1}[\Phi]\otimes_{\mathcal{O}_{1}[\Phi^{d}]}M_{2}$
and to calculate the resulting actions of $\mathbf{\widetilde{A}}_{L_{1}}$,
$\varphi_{1}$ and $\Gamma_{1}$.

Write $\widehat{\mathbf{A}}=W(\mathbb{C}_{p}^{\flat}).$ This ring
does not depend on $L$ and $\widehat{\mathbf{A}}_{i}=\mathcal{O}_{i}\otimes_{W(\kappa_{i})}\widehat{\mathbf{A}}$
~($i=1,2).$ We get
\[
(\widehat{\mathbf{A}}_{1}\otimes_{\mathcal{O}_{1}}V_{2})^{H_{2}}=(\widehat{\mathbf{A}}\otimes_{W(\kappa_{1})}V_{2})^{H_{2}}=(\widehat{\mathbf{A}}\otimes_{W(\kappa_{2})}[W(\kappa_{2})\otimes_{W(\kappa_{1})}W(\kappa_{2})]\otimes_{W(\kappa_{2})}V_{2})^{H_{2}}.
\]
Now use the isomorphism
\[
W(\kappa_{2})\otimes_{W(\kappa_{1})}W(\kappa_{2})\simeq W(\kappa_{2})^{d},\,\,\,\,\,a\otimes b\mapsto(a\varphi_{1}^{i}(b))_{i=0}^{d-1}
\]
to identify the module $(\widehat{\mathbf{A}}_{1}\otimes_{\mathcal{O}_{1}}V_{2})^{H_{2}}$
with the direct sum of $d$ copies of $M_{2}=(\widehat{\mathbf{A}}_{2}\otimes_{\mathcal{O}_{2}}V_{2})^{H_{2}}=(\widehat{\mathbf{A}}\otimes_{W(\kappa_{2})}V_{2})^{H_{2}}$
as desired.

We leave as an exercise tracing the various identifications and verifying
that the resulting actions of $\mathbf{\widetilde{A}}_{L_{1}},$ $\Gamma_{1}$
and $\varphi_{1}$ are as indicated in the definition of $\mathcal{I}_{L_{1}}^{L_{2}}(M_{2}).$
The non-obvious action of $\mathbf{\widetilde{A}}_{L_{1}}$ results
from the replacement of $f$ with $f^{\sharp}.$
\end{proof}

\section{The ring $\mathbf{\widetilde{A}}_{L}^{\dagger}$ and overconvergence}

\subsection{The rings of overconvergent periods}

We recall the definition of the subring $\mathbf{\widetilde{A}}_{L}^{\dagger}$
(resp. $\mathbf{\widetilde{A}}^{\dagger}$) of $\mathbf{\widetilde{A}}_{L}$
(resp. $\mathbf{\widetilde{A}}$) consisting of overconvergent periods.
This will still be a discrete valuation ring with uniformizer $\pi$
and residue field $F$ (resp. $F^{sep}$), dense in $\mathbf{\widetilde{A}}_{L}$
(resp. $\mathbf{\widetilde{A}}^{\dagger}$) in the weak topology.
The automorphism $\varphi$ and the action of $\Gamma$ (resp. $G$)
will be induced by the corresponding actions on $\mathbf{\widetilde{A}}_{L}$
(resp. $\mathbf{\widetilde{A}}$).

To define this ring introduce, for $0\le r\le\infty$ and $x\in\mathbf{\widetilde{A}}_{L}$,
a ``norm'' $|x|_{r}\ge0,$ which nevertheless may be infinite. Write
$x=\sum_{n=0}^{\infty}\pi^{n}\tau(\xi_{n})$ with $\xi_{n}\in F$
and let, if $r<\infty$
\[
|x|_{r}=\sup_{n}\{q^{-n}|\xi_{n}|_{\flat}^{r}\}.
\]
If $r=0$, $|\xi|_{\flat}^{r}=1$ if $\xi\ne0$ and is $0$ if $\xi=0$.
Thus $|x|_{0}$ is always finite and is nothing but the norm associated
with the discrete valuation. When $0\le s<r$ we have
\[
|x|_{s}\le|x|_{r}^{s/r}
\]
so if $|x|_{r}<\infty$ also $|x|_{s}<\infty$. Finally if $r=\infty$
we put
\[
|x|_{\infty}=\sup_{n}\{|\xi_{n}|_{\flat}\},
\]
so that $|x|_{r}\le|x|_{\infty}^{r}.$ Note that if $x\in W(\mathcal{O}_{F})_{L}$
then $|x|_{r}<\infty$ for any $0\le r\le\infty.$

It is not hard to see (\cite{=00005BKe=00005D}, Lemma 1.7.2) that
each $|\cdot|_{r}$ is a multiplicative norm on the subring $\mathbf{\widetilde{A}}_{L}^{(0,r)}$
of $x\in\mathbf{\widetilde{A}}_{L}$ for which $|x|_{r}<\infty$.
This subring grows when $r$ decreases and we let $\mathbf{\widetilde{A}}_{L}^{\dagger}$
be the ring of all $x\in\mathbf{\widetilde{A}}_{L}$ for which there
exists an $r>0$ with $|x|_{r}<\infty,$ i.e.
\[
\mathbf{\widetilde{A}}_{L}^{\dagger}=\bigcup_{r>0}\mathbf{\widetilde{A}}_{L}^{(0,r)}.
\]

The ring $\mathbf{\widetilde{A}}^{\dagger}$ is defined in precisely
the same manner, using the extension of the norm $|\xi|_{\flat}$
to $\xi\in F^{sep}.$ It is clear that
\[
\pi^{n}\mathbf{\widetilde{A}}_{L}^{\dagger}=\mathbf{\widetilde{A}}_{L}^{\dagger}\cap\pi^{n}\mathbf{\widetilde{A}}_{L}.
\]
We claim that $\mathbf{\widetilde{A}}_{L}^{\dagger}$ is a discrete
valuation ring with $\pi$ as a prime. To see this it is enough to
check that if $x\in\mathbf{\widetilde{A}}_{L}^{\dagger}$ is invertible
in $\mathbf{\widetilde{A}}_{L}$ then its inverse $x^{-1}\in\mathbf{\widetilde{A}}_{L}^{\dagger}$.
This follows at once from the multiplicativity of the norm $|\cdot|_{r}.$
The residue field of $\mathbf{\widetilde{A}}_{L}^{\dagger}$ is contained
in the residue field of $\mathbf{\widetilde{A}}_{L}$, i.e. in $F,$
and contains $W(\mathcal{O}_{F})_{L}/\pi W(\mathcal{O}_{F})_{L}=\mathcal{O}_{F}$,
hence must be equal to $F.$ Similar arguments apply to $\widetilde{\mathbf{A}}^{\dagger}$. 

The automorphism $\varphi$ intertwines $|\cdot|_{r}$ and $|\cdot|_{qr}$,
while $\Gamma$ or $G$ preserve $|\cdot|_{r}$. This implies that
$\mathbf{\widetilde{A}}_{L}^{\dagger}$ and $\widetilde{\mathbf{A}}^{\dagger}$
are invariant under $\varphi$ and the $\Gamma$ or $G$ action, which
are continuous in the weak topology.

\bigskip{}

We conclude by giving an alternative description of $\mathbf{\widetilde{A}}_{L}^{\dagger}$
which explains the name \emph{overconvergent} periods. For $0\le r<\infty$
let
\[
R_{r}=q^{-rq/(q-1)}.
\]
As $r\to0$ from above, $R_{r}\to1$ from below.
\begin{prop}
Let $x\in\mathbf{\widetilde{A}}_{L}$ correspond to the power series
$f=\sum_{m\in\mathbb{Z}[p^{-1}]}a_{m}X^{m}$ ($a_{m}\in W_{L}$) under
the isomorphism of Proposition \ref{prop:Another look at A_L}. Then
for $0\le r<\infty$
\[
|x|_{r}=||f||_{r}:=\sup_{m}\{|a_{m}|R_{r}^{m}\}.
\]
 
\end{prop}

\begin{proof}
Let $x=\sum_{n=0}^{\infty}\pi^{n}\tau(\xi_{n})$ and $f=\sum_{n=0}^{\infty}\sum_{m\in\mathbb{Z}[p^{-1}]}\pi^{n}\tau(\alpha_{n,m})X^{m}$
where $\alpha_{n,m}\in\overline{\kappa}$ are the ``digits'' of
$a_{m}.$ Clearly $||f||_{r}$ is the supremum of $q^{-n}R_{r}^{m}$
over the pairs $(n,m)$ such that $\alpha_{n,m}\ne0.$ Reduction modulo
$\pi$ shows that $\xi_{0}=\sum_{m}\alpha_{0,m}\omega^{m}$ hence
\[
|\tau(\xi_{0})|_{r}=|\xi_{0}|_{\flat}^{r}=\sup\{R_{r}^{m}\}
\]
where the sup is over the set of $m$ such that $\alpha_{0,m}\ne0.$
But this is also
\[
|\tau(\xi_{0})|_{r}=||f_{0}||_{r}
\]
where $f_{0}=\sum_{m\in\mathbb{Z}[p^{-1}]}\tau(\alpha_{0,m})X^{m}.$
Consider
\[
x'=x-\tau(\xi_{0}),\,\,f''=f-f_{0}.
\]
Then
\[
|x|_{r}=\sup\{|\tau(\xi_{0})|_{r},|x'|_{r}\},\,\,\,||f||_{r}=\sup\{||f_{0}||_{r},||f''||_{r}\}.
\]
The power series corresponding to $x'$ is not $f''$ but rather $f'=f-\tau(\xi_{0}).$
However,
\[
||f_{0}-\tau(\xi_{0})||_{r}\le||f_{0}||_{r}
\]
so $||f||_{r}=\sup\{||f_{0}||_{r},||f'||_{r}\}$ as well. We may now
divide both $x'$ and $f'$ by $\pi$ and continue recursively to
get the desired equality $|x|_{r}=||f||_{r}.$
\end{proof}
Let $V(R,1)=\{X|\,R<|X|<1\}$, regarded as a rigid analytic annulus
over (the fraction field of) $W_{L}$. Taking inverse limit with respect
to $X\mapsto X^{p}$ gives
\[
\widetilde{V}(R,1)=\lim_{\leftarrow}\left(V(R,1)\overset{p}{\leftarrow}V(R^{1/p},1)\overset{p}{\leftarrow}\cdots\right),
\]
which we regard as a \emph{preperfectoid space }in the sense of \cite{=00005BScho-We=00005D}.
A point of $\widetilde{V}(R,1)$ in some analytic field containing
$L$ amounts to giving compatible values to $X^{m}$, for every $m\in\mathbb{Z}[p^{-1}].$
\begin{cor}
(i) $\mathbf{\widetilde{A}}_{L}^{\dagger}$ consists of those $f\in\mathbf{\widetilde{A}}_{L}$
which converge on $\widetilde{V}(R,1)$ for some $R<1$.

(ii) For $0<r<\infty$ the ring $\mathbf{\widetilde{A}}_{L}^{(0,r)}$
consists of those $f\in\mathbf{\widetilde{A}}_{L}$ which converge
and are bounded on $\widetilde{V}(R_{r},1)$.

(iii) The ring $\mathbf{\widetilde{A}}_{L}^{(0,\infty)}$ consists
of those $f\in\mathbf{\widetilde{A}}_{L}$ which converge on $\widetilde{V}(0,1)$
and have a pole at $0.$
\end{cor}

\begin{proof}
Write $f=f^{+}+f^{-}$ where $f^{+}=\sum_{0\le m\in\mathbb{Z}[p^{-1}]}a_{m}X^{m}$
and $f^{-}=\sum_{0>m\in\mathbb{Z}[p^{-1}]}a_{m}X^{m}.$ As $a_{m}\in W_{L}$,
$f^{+}$ is convergent on the open unit disk and $||f^{+}||_{r}\le1,$
while if $r<\infty$ $f^{-}$ is convergent on $\widetilde{V}(R_{r},1)$
and bounded there if and only if $||f^{-}||_{r}<\infty.$ For (iii)
note that $f\in\mathbf{\widetilde{A}}_{L}^{(0,\infty)}$ if and only
if for some $n$, $X^{n}f\in W(\mathcal{O}_{F})_{L}.$
\end{proof}
The ring of all bounded rigid analytic functions on some $\widetilde{V}(R,1)$
is the \emph{bounded (preperfectoid) Robba ring} $\widetilde{\mathbf{B}}_{L}^{\dagger}=\mathbf{\widetilde{A}}_{L}^{\dagger}[p^{-1}]$.
It is actually a field, the fraction field of $\mathbf{\widetilde{A}}_{L}^{\dagger}$,
but we shall have no use for it.

\subsection{Overconvergent $(\varphi,\Gamma)$-modules}

The definition of an overconvergent $(\varphi,\Gamma)$-module is
the same as the one of a formal $(\varphi,\Gamma)$-module, substituting
the ring $\mathbf{\widetilde{A}}_{L}^{\dagger}$ for $\mathbf{\widetilde{A}}_{L}$.
It is sometimes customary to impose in the definition a further continuity
condition on the action of $\Gamma$, on which we comment now. In
addition to the weak topology inherited from $\mathbf{\widetilde{A}}_{L}$,
the ring $\mathbf{\widetilde{A}}_{L}^{\dagger}$ has a ``limit of
Fréchet'' (LF) topology, resulting from the Fréchet topologies of
uniform convergence on affinoid sub-annuli on each $\mathbf{\widetilde{A}}_{L}^{(0,r)}$.
Just as for the weak topology, every finitely generated $\mathbf{\widetilde{A}}_{L}^{\dagger}$-module
is endowed then with a canonical LF topology, and one requires the
action of $\Gamma$ to be continuous in it as well. For étale $(\varphi,\Gamma)$-modules
over $\mathbf{\widetilde{A}}_{L}^{\dagger}$, continuity of the action
of $\Gamma$ in the weak topology most probably implies its continuity
in the LF topology (see Lemma 2.4.3 of \cite{=00005BKe=00005D}, under
the assumption that the module is trivial modulo $p$, an assumption
that ought to be irrelevant). We therefore do not impose continuity
in the LF topology as part of our definition.

We denote by
\[
\mathrm{Mod}{}_{\varphi,\Gamma}^{\dagger\acute{e}t}(L)\subset\mathrm{Mod}_{\varphi,\Gamma}^{\dagger}(L)
\]
the category of overconvergent $(\varphi,\Gamma)$-modules, and its
full subcategory of overconvergent étale $(\varphi,\Gamma)$-modules.
The Cherbonnier-Colmez theorem, in our setting, is the following.
\begin{thm}
\label{thm:o/c}(i) Base change from $\mathbf{\widetilde{A}}_{L}^{\dagger}$
to $\mathbf{\widetilde{A}}_{L}$ induces an equivalence of categories
between $\mathrm{Mod}_{\varphi,\Gamma}^{\dagger\acute{e}t}(L)$ and
$\mathrm{Mod}_{\varphi,\Gamma}^{\acute{e}t}(L)$.

(ii) Let $V\in\mathrm{Rep}_{\mathcal{O}}(L)$ and put $\mathcal{D}^{\dagger}(V):=(\mathbf{\widetilde{A}}_{L}^{\dagger}\otimes_{\mathcal{O}}V)^{H}.$
Then $\mathcal{D}^{\dagger}(V)\in\mathrm{Mod}_{\varphi,\Gamma}^{\dagger\acute{e}t}(L)$
and $\mathcal{D}(V)$ is the base change from $\mathbf{\widetilde{A}}_{L}^{\dagger}$
to $\mathbf{\widetilde{A}}_{L}$ of $\mathcal{D}^{\dagger}(V).$
\end{thm}

See {[}Ke{]}, Theorem 2.4.5. The structure of $\Gamma$ is irrelevant
in the proof of that theorem (unlike the proof of \emph{loc.cit. }Theorem
2.6.2) so although it is phrased in the cyclotomic setting, it works
for $\Gamma=Gal(L^{ab}/L)$ as well. Also, {[}Ke{]} treats only torsion-free
representations and torsion-free modules, but our claim follows from
this easily, since $\mathbf{\widetilde{A}}_{L}^{\dagger}/\pi^{n}\mathbf{\widetilde{A}}_{L}^{\dagger}=\mathbf{\widetilde{A}}_{L}/\pi^{n}\mathbf{\widetilde{A}}_{L},$
so torsion modules are the same for the two rings.

As explained in the introduction, the Cherbonnier-Colmez theorem \emph{fails}
with $\mathbf{A}_{L}$ and $\mathbf{A}_{L}^{\dagger}$ replacing $\mathbf{\widetilde{A}}_{L}$
and $\mathbf{\widetilde{A}}_{L}^{\dagger}$ whenever $L\ne\mathbb{Q}_{p}.$

For completeness we give a self-contained proof of the theorem, based
on Lemma \ref{lem: estimate}. Kedlaya uses a similar estimate, but
only for $(\varphi,\Gamma)$-modules which are trivial modulo $p,$
something that can be achieved (in view of Theorem \ref{thm:Equivalence})
after restriction to a finite Galois extension $L'$ of $L$. He then
ends up using Galois descent to go back from $L'$ to $L$. We believe
that our proof is a little more transparent.

The notion of a $\varphi$-module over the ring $\widetilde{\mathbf{A}}_{L}=\mathcal{O}\otimes_{W(\kappa)}W(F)$
is defined as before, without any reference to the action of $\Gamma$,
and in fact makes sense (over $W(F)$) for any perfectoid field $F$
in characteristic $p$, whether realized as $K^{\flat}$ for some
characteristic $0$ perfectoid field $K$ or not. A $\varphi$-module
$M$ is étale if $\varphi$ is bijective. Similarly, one defines the
notion of an overconvergent (general or étale) $\varphi$-module over
$\widetilde{\mathbf{A}}_{L}^{\dagger}.$
\begin{lem}
\label{lem:M_dagger}Let $M$ be an étale $\varphi$-module over $\widetilde{\mathbf{A}}_{L}$.
Then there exists an overconvergent étale $\varphi$-module $M^{\dagger}$over
$\widetilde{\mathbf{A}}_{L}^{\dagger}$, contained in $M$, such that
the canonical map
\[
\widetilde{\mathbf{A}}_{L}\otimes_{\widetilde{\mathbf{A}}_{L}^{\dagger}}M^{\dagger}\to M
\]
is an isomorphism.
\end{lem}

\begin{rem*}
We actually prove a stronger statement, that there exists an étale
$\varphi$-module $M^{(0,\infty)}$ over $\widetilde{\mathbf{A}}_{L}^{(0,\infty)}$
for which $\widetilde{\mathbf{A}}_{L}\otimes_{\widetilde{\mathbf{A}}_{L}^{(0,\infty)}}M^{(0,\infty)}\to M$
is an isomorphism. Note that $\widetilde{\mathbf{A}}_{L}^{(0,\infty)}$
is stable under $\varphi,$ but $\widetilde{\mathbf{A}}_{L}^{(0,r)}$
for $0<r<\infty$ is not.
\end{rem*}
\begin{proof}
Since $\mathbf{\widetilde{A}}_{L}^{\dagger}/\pi^{n}\mathbf{\widetilde{A}}_{L}^{\dagger}=\mathbf{\widetilde{A}}_{L}/\pi^{n}\mathbf{\widetilde{A}}_{L},$
if $M$ is torsion we can take $M^{\dagger}=M.$ Suppose we prove
the Lemma when $M$ is torsion-free. We can then consider the exact
sequence
\[
0\to M_{tor}\to M\overset{pr}{\to}N\to0
\]
where $N$ is torsion-free and let $M^{\dagger}=pr^{-1}(N^{\dagger}).$
We may therefore assume that $M$ is freely generated by $e_{1},\dots,e_{d}$
over $\mathbf{\widetilde{A}}_{L}$ and that the matrix of $\varphi$
in this basis is given by $A=(a_{ij})\in GL_{d}(\mathbf{\widetilde{A}}_{L}),$
i.e.
\[
\varphi(e_{j})=\sum_{i=1}^{d}a_{ij}e_{i}.
\]
If $U=(u_{ij})\in GL_{d}(\mathbf{\widetilde{A}}_{L})$ and $e_{j}'=\sum_{i=1}^{d}u_{ij}e_{i}$
then the matrix of $\varphi$ in the basis $\{e_{j}'\}$ is $U^{-1}A\varphi(U).$
Our goal is to find $U$ such that $C=U^{-1}A\varphi(U)\in GL_{d}(\mathbf{\widetilde{A}}_{L}^{\dagger}),$
because then we can take $M^{\dagger}=\sum_{j=1}^{d}\mathbf{\widetilde{A}}_{L}^{\dagger}e_{j}'$.
We shall in fact find a $U$ such that $C\in GL_{d}(\mathbf{\widetilde{A}}_{L}^{(0,\infty)})$
(see the remark preceding the proof).

Let us write $A=\sum_{n=0}^{\infty}\pi^{n}\tau(A_{n})$ where $A_{n}\in M_{d}(F)$
and $\tau(A_{n})$ is the matrix obtained by taking the Teichmüller
representatives of the entries of $A_{n}$ one-by-one. Note that $A_{0}\in GL_{d}(F).$
Similarly write $U=\sum_{n=0}^{\infty}\pi^{n}\tau(U_{n})$ and
\[
C=U^{-1}A\varphi(U)=\sum_{n=0}^{\infty}\pi^{n}\tau(C_{n}).
\]
It is enough to construct $U$ so that $|C_{n}|_{\flat}$ is bounded,
as the entries of $C$ will then lie in $\mathbf{\widetilde{A}}_{L}^{(0,\infty)}.$
Recall that the norm of a matrix with entries from $F$ is defined
to be the maximum of the norms of its entries.

Let $U_{0}=I$ and suppose $U_{0},\dots,U_{n-1}$ $(n\ge1)$ have
been defined. Let $U'=\sum_{i=0}^{n-1}\pi^{i}\tau(U_{i})$, and $C_{0}=A_{0},C_{1},\dots,C_{n-1}\in M_{d}(F)$
the matrices such that
\[
U'^{-1}A\varphi(U')\equiv\sum_{i=0}^{n-1}\pi^{i}\tau(C_{i})\mod\pi^{n}.
\]
Write $U'^{-1}A\varphi(U')-\sum_{i=0}^{n-1}\pi^{i}\tau(C_{i})=\pi^{n}B$
and look for $U_{n}\in M_{d}(F)$ so that
\[
(U'+\pi^{n}\tau(U_{n}))^{-1}A\varphi(U'+\pi^{n}\tau(U_{n}))\equiv\sum_{i=0}^{n}\pi^{i}\tau(C_{i})\mod\pi^{n+1}
\]
with $|C_{n}|_{\flat}$ small. If we denote by $\overline{B}\in M_{d}(F)$
the reduction of $B$ modulo $\pi$, the above equation is equivalent
to
\[
U_{n}-A_{0}\varphi(U_{n})A_{0}^{-1}=\overline{B}A_{0}^{-1}-C_{n}A_{0}^{-1}.
\]
Lemma \ref{lem: estimate} guarantees that $U_{n}$ can be chosen
so that $|C_{n}A_{0}^{-1}|_{\flat},$ hence also $|C_{n}|_{\flat}$,
is bounded uniformly in $n.$ The bound depends only on $A_{0}$.
This concludes the induction step, and with it the proof of the Lemma.
\end{proof}
The next lemma is a manifestation of the ``contracting'' property
of Frobenius.
\begin{lem}
Let $A\in GL_{d}(\mathbf{\widetilde{A}}_{L}^{\dagger}),\,\,B\in GL_{e}(\mathbf{\widetilde{A}}_{L}^{\dagger})$
and $U\in M_{d\times e}(\mathbf{\widetilde{A}}_{L})$ satisfy
\[
A\cdot\varphi(U)=U\cdot B.
\]
Then $U\in M_{d\times e}(\mathbf{\widetilde{A}}_{L}^{\dagger}).$
\end{lem}

\begin{rem*}
If $A$ and $B$ have entries in $\mathbf{\widetilde{A}}_{L}^{(0,r)}$
for some $0<r\le\infty,$ the proof will show that so does $U.$
\end{rem*}
\begin{proof}
Let $0<r\le\infty$ be such that $A$ and $B$ have entries in $\mathbf{\widetilde{A}}_{L}^{(0,r)}.$
We let $|B|_{r}$ be the maximum of $|b_{ij}|_{r}$ where $b_{ij}$
are the entries of $B$, and similarly for $A$. Suppose that $|B|_{r}\le1.$
Then for $i\ge1$ $|\varphi^{-i}(B)|_{r}\le1$, and for any integer
$N$
\[
|\varphi^{-N}(B)\cdots\varphi^{-2}(B)\varphi^{-1}(B)|_{r}\le|\varphi^{-N}(B)|_{r}\cdots|\varphi^{-2}(B)|_{r}|\varphi^{-1}(B)|_{r}\le1.
\]
In general, $B$ can be written as $\tau(\beta)B_{0}$ with $\beta\in F$
and $|B_{0}|_{r}\le1$, so (if $r<\infty$)
\[
|\varphi^{-N}(B)\cdots\varphi^{-2}(B)\varphi^{-1}(B)|_{r}\le|\beta|_{\flat}^{r(q^{-1}+\cdots+q^{-N})}
\]
is bounded independently of $N$, and similarly (with another bound)
if $r=\infty.$

From the equation $U=(\varphi^{-1}(A))^{-1}\cdot\varphi^{-1}(U)\cdot\varphi^{-1}(B)$
we get by iteration
\[
U=(\varphi^{-1}(A))^{-1}(\varphi^{-2}(A))^{-1}\cdots(\varphi^{-N}(A))^{-1}\cdot\varphi^{-N}(U)\cdot\varphi^{-N}(B)\cdots\varphi^{-2}(B)\varphi^{-1}(B),
\]
which we write as $U=A_{N}^{-1}\varphi^{-N}(U)B_{N}$ with $|A_{N}^{-1}|_{r}$
and $|B_{N}|_{r}$ bounded independently of $N$. Let $U^{(n)}$ be
the truncation of $U$ modulo $\pi^{n},$ i.e. if $U=\sum_{i=0}^{\infty}\pi^{i}\tau(U_{i})$,
then $U^{(n)}=\sum_{i=0}^{n-1}\pi^{i}\tau(U_{i})$. We then have
\[
U^{(n)}\equiv A_{N}^{-1}\cdot\varphi^{-N}(U^{(n)})\cdot B_{N}\mod\pi^{n}.
\]
Fixing $n$ and choosing $N$ large we can make $|A_{N}^{-1}\cdot\varphi^{-N}(U^{(n)})\cdot B_{N}|_{r}\le c$
where $c$ is a constant depending only of $A$ and $B$. But if $V=\sum_{i=0}^{\infty}\pi^{i}\tau(V{}_{i})$
is a matrix then
\[
|V|_{r}=\sup_{i}\{q^{-i}|V_{i}|_{\flat}^{r}\}
\]
if $r<\infty$ and $|V|_{\infty}=\sup_{i}\{|V_{i}|_{\flat}\}$, so
clearly $|V^{(n)}|_{r}\le|V|_{r}$. We conclude that $|U^{(n)}|_{r}\le c.$
As this is true for every $n,$ $|U|_{r}\le c<\infty,$ as was to
be shown.
\end{proof}
\begin{cor}
\label{cor:uniqueness}The submodule $M^{\dagger}$ whose existence
is guaranteed by Lemma \ref{lem:M_dagger} is unique.
\end{cor}

\begin{proof}
As usual we may assume that $M$ is torsion-free. Let $M^{\dagger}$
be freely generated over $\mathbf{\widetilde{A}}_{L}^{\dagger}$ by
$e_{1},\dots,e_{d}$ and suppose $e_{j}'=\sum_{i=1}^{d}u_{ij}e_{i}$
span another $\varphi$-stable $\mathbf{\widetilde{A}}_{L}^{\dagger}$-submodule
where $U=(u_{ij})\in GL_{d}(\mathbf{\widetilde{A}}_{L}).$ Our goal
is to show that $U\in GL_{d}(\mathbf{\widetilde{A}}_{L}^{\dagger}).$
Letting $A$ and $B$ be the matrices from $GL_{d}(\mathbf{\widetilde{A}}_{L}^{\dagger})$
expressing $\varphi$ in the bases $\{e_{i}\}$ and $\{e'_{i}\}$
respectively, we get the relation $A\cdot\varphi(U)=U\cdot B,$ hence
the corollary follows from the lemma.
\end{proof}
\begin{cor}
\label{cor:phi_invariants}Let $M$ and $M^{\dagger}$ be as in Lemma
\ref{lem:M_dagger}. Then
\[
M^{\varphi}=(M^{\dagger})^{\varphi}.
\]
\end{cor}

\begin{proof}
We first reduce to the case where $M$ is torsion free. Let $N=M/M_{tor}$
and $N^{\dagger}=M^{\dagger}/M_{tor}^{\dagger}.$ Recall that $M_{tor}^{\dagger}=M_{tor}$.
Consider the commutative diagram

\[
\xymatrix{0\ar[r] & M_{tor}^{\dagger\varphi}\ar[r]\ar@{=}[d] & M^{\dagger\varphi}\ar[r]\ar[d] & N^{\dagger\varphi}\ar[r]\ar[d] & M_{tor}^{\dagger}/(\varphi-1)\ar@{=}[d]\\
0\ar[r] & M_{tor}^{\varphi}\ar[r] & M^{\varphi}\ar[r] & N^{\varphi}\ar[r] & M_{tor}/(\varphi-1)
}
\]
whose rows are exact. If $N^{\varphi}=(N^{\dagger})^{\varphi}$ then
the same would hold with $M.$

Assume therefore that $M$ is torsion-free, let $e_{1},\dots,e_{d}$
be a basis of $M^{\dagger}$ and $A\in GL_{d}(\mathbf{\widetilde{A}}_{L}^{\dagger})$
the matrix of $\varphi$ in this basis. If $m=\sum_{i=1}^{d}u_{i}e_{i}\in M^{\varphi}$
the coordinate vector $u$ satisfies
\[
A\cdot\varphi(u)=u
\]
so by the lemma $u_{i}\in\mathbf{\widetilde{A}}_{L}^{\dagger}$ and
$m\in(M^{\dagger})^{\varphi}.$
\end{proof}
We can summarize the discussion so far in the following proposition.
\begin{prop}
\label{prop:o/c for phi modules}(i) Let $\mathrm{Mod}_{\varphi}(L)$
be the category of $\varphi$-modules over $\mathbf{\widetilde{A}}_{L}$
and $\mathrm{Mod}_{\varphi}^{\dagger}(L)$ the category of overconvergent
$\varphi$-modules over $\mathbf{\widetilde{A}}_{L}^{\dagger}$. Then
the functors $M^{\dagger}\mapsto M:=\mathbf{\widetilde{A}}_{L}\otimes_{\mathbf{\widetilde{A}}_{L}^{\dagger}}M^{\dagger}$
and $M\mapsto M^{\dagger}$ (given by Lemma \ref{lem:M_dagger}) induce
an equivalence of categories between the full subcategories $\mathrm{Mod}_{\varphi}^{\acute{e}t}(L)$
and $\mathrm{Mod}_{\varphi}^{\dagger\acute{e}t}(L)$ of étale submodules.

(ii) Let $V\in\mathrm{Rep}_{\mathcal{O}}(F)$ be a continuous representation
of $H=Gal(F^{sep}/F)$ on a finitely generated $\mathcal{O}$-module
$V$. If $M=\mathcal{D}(V):=(\mathbf{\widetilde{A}}\otimes_{\mathcal{O}}V)^{H}$
is the corresponding étale $\varphi$-module, then $M^{\dagger}=\mathcal{D}^{\dagger}(V):=(\mathbf{\widetilde{A}^{\dagger}}\otimes_{\mathcal{O}}V)^{H}$.
\end{prop}

\begin{proof}
(i) Lemma \ref{lem:M_dagger} and Corollary \ref{cor:uniqueness}
imply that the two functors induce a bijection between the objects
of the two categories. The categories $\mathrm{Mod}_{\varphi}^{\acute{e}t}(L)$
and $\mathrm{Mod}_{\varphi}^{\dagger\acute{e}t}(L)$ have tensor products
and internal Hom's (for the latter we need $\varphi$ to be bijective!)
and the two functors commute with them. Since
\[
\mathrm{Hom_{\varphi}}(M,N)=\underline{Hom}(M,N)^{\varphi}
\]
(the first Hom is the group of morphisms in the category, the second
underlined Hom is the internal Hom), and since the same holds in the
overconvergent category, the equality $\mathrm{Hom_{\varphi}}(M,N)=\mathrm{Hom_{\varphi}}(M^{\dagger},N^{\dagger})$
follows from the obvious fact that $\underline{Hom}(M,N)^{\dagger}=\underline{Hom}(M^{\dagger},N^{\dagger})$
and from Corollary \ref{cor:phi_invariants}.

(ii) Lemma \ref{lem:M_dagger} and Corollary \ref{cor:uniqueness}
remain valid, with the same proof, for finitely generated étale $\varphi$-modules
$X$ over $\mathbf{\widetilde{A}}$ and their overconvergent variants
over $\mathbf{\widetilde{A}}^{\dagger}$. Consider
\[
X=\mathbf{\widetilde{A}}\otimes_{\mathbf{\widetilde{A}}_{L}}M=\mathbf{\widetilde{A}}\otimes_{\mathcal{O}}V.
\]
 Both $\mathbf{\widetilde{A}^{\dagger}}\otimes_{\mathbf{\widetilde{A}}_{L}^{\dagger}}M^{\dagger}$
and $\mathbf{\widetilde{A}}^{\dagger}\otimes_{\mathcal{O}}V$ are
finitely generated $\mathbf{\widetilde{A}}^{\dagger}$-submodules
$X^{\dagger}$ of $X$ satisfying
\[
\mathbf{\widetilde{A}}\otimes_{\mathbf{\widetilde{A}}^{\dagger}}X^{\dagger}=X.
\]
By the uniqueness of $X^{\dagger},$ they are equal: $\mathbf{\widetilde{A}^{\dagger}}\otimes_{\mathbf{\widetilde{A}}_{L}^{\dagger}}M^{\dagger}=\mathbf{\widetilde{A}}^{\dagger}\otimes_{\mathcal{O}}V$.
Taking $H$-invariants we get the desired formula
\[
M^{\dagger}=\mathcal{D}^{\dagger}(V).
\]
\end{proof}
We can now conclude the proof of Theorem \ref{thm:o/c}.
\begin{proof}
(of Theorem \ref{thm:o/c}) Assume now that $M^{\dagger}\in\mathrm{Mod}{}_{\varphi,\Gamma}^{\dagger\acute{e}t}(L).$
Then $M=\mathbf{\widetilde{A}}_{L}\otimes_{\mathbf{\widetilde{A}}_{L}^{\dagger}}M^{\dagger}$
carries a semi-linear continuous action of $\Gamma$ commuting with
$\varphi$, hence $M\in\mathrm{Mod}{}_{\varphi,\Gamma}^{\acute{e}t}(L).$
Conversely, if $M\in\mathrm{Mod}{}_{\varphi,\Gamma}^{\acute{e}t}(L)$
and $M^{\dagger}$ is the $\varphi$-submodule constructed in Lemma
\ref{lem:M_dagger}, then for any $\gamma\in\Gamma$ the $\varphi$-submodule
$\gamma(M^{\dagger})$ also satisfies the conditions of the Lemma,
so by uniqueness (Corollary \ref{cor:uniqueness}) $M^{\dagger}$
is stable under $\Gamma.$ The action of $\Gamma$ on $M^{\dagger}$
is clearly continuous in the weak topology inherited from $M$, so
$M^{\dagger}\in\mathrm{Mod}{}_{\varphi,\Gamma}^{\dagger\acute{e}t}(L).$
Finally, the morphisms between two $(\varphi,\Gamma)$-modules are
the morphisms in the category of $\varphi$-modules, which furthermore
commute with $\Gamma$. The same applies to overconvergent $(\varphi,\Gamma)$-modules.
If $M$ and $N$ are étale we have already shown that $\mathrm{Hom}_{\varphi}(M,N)=\mathrm{Hom}_{\varphi}(M^{\dagger},N^{\dagger})$,
hence we also have $\mathrm{Hom}_{\varphi,\Gamma}(M,N)=\mathrm{Hom}_{\varphi,\Gamma}(M^{\dagger},N^{\dagger})$.

Part (ii) of the theorem follows from part (ii) of Proposition \ref{prop:o/c for phi modules},
because the functors $\mathcal{D}$ and $\mathcal{D}^{\dagger}$ do
not involve the $\Gamma$-action.
\end{proof}

\end{document}